\documentclass[11pt,a4paper]{article}

\usepackage[utf8]{inputenc}
\usepackage[T1]{fontenc}
\usepackage[margin=1.2in]{geometry}
\usepackage{lmodern}

\usepackage{mathtools}
\usepackage{amsthm}
\usepackage{todonotes}
\usepackage{mathrsfs,amssymb}  

\usepackage{enumitem}
\setlist[description]{font=\mdseries}
\setlist[enumerate,1]{label=\textup{\arabic*.},ref=\textup{\arabic*}}
\setlist[enumerate,2]{label=\textup{(\alph*)},ref=\textup{(\alph*)}}
\setlist[enumerate,3]{label=\textup{(\roman*)},ref=\textup{(\roman*)}}

\newtheorem{theorem}{Theorem}

\newtheorem{claim}[theorem]{Claim}

\newtheorem{proposition}[theorem]{Proposition}
\newtheorem{lemma}[theorem]{Lemma}
\newtheorem{corollary}[theorem]{Corollary}
\newtheorem{conjecture}[theorem]{Conjecture}

\theoremstyle{definition}
\newtheorem*{definition*}{Definition}

\newcommand*{\myproofname}{Proof}
\newenvironment{claimproof}[1][\myproofname]{\begin{proof}[#1]}{\end{proof}}

 \newcommand{\Exp}{\,\mathbb{E}}
 \renewcommand{\Pr}{\,\mathbb{P}}

\DeclareMathOperator{\s}{s}

\DeclareMathOperator{\Per}{Per}
\DeclareMathOperator{\PP}{PP}

\let\le\leqslant
\let\ge\geqslant
\let\leq\leqslant

\let\geq\geqslant

\usepackage{bookmark}

\newcommand*{\eps}{\varepsilon}
\renewcommand*{\phi}{\varphi}
\renewcommand*{\emptyset}{\varnothing}
\newcommand*{\indicator}[1]{\mathbf{1}\{#1\}}
\newcommand*{\EE}{\mathbb{E}}

\newcommand*{\rr}{r}

\def\P{\mathcal{P}}

\begin{document}

\title{Packing list-colourings}

\author{
	Stijn Cambie%
\thanks{Extremal Combinatorics and Probability Group (ECOPRO), Institute for Basic Science (IBS), Daejeon, South Korea. Partially supported by IBS-R029-C4 and a Vidi grant (639.032.614) of the Netherlands Organisation for Scientific Research (NWO). Email: \protect\href{mailto:stijn.cambie@hotmail.com}{\protect\nolinkurl{stijn.cambie@hotmail.com}}.}
	\and
	Wouter Cames van Batenburg%
	\thanks{Delft Institute of Applied Mathematics, Delft University of Technology, Netherlands.
		Email: \protect\href{mailto:w.p.s.camesvanbatenburg@tudelft.nl}{\protect\nolinkurl{w.p.s.camesvanbatenburg@tudelft.nl}}. Partially supported by ANR Project GATO (ANR-16-CE40-0009-01).}
	\and
	Ewan Davies%
	\thanks{Department of Computer Science, Colorado State University, USA.
		Email: \protect\href{mailto:research@ewandavies.org}{\protect\nolinkurl{research@ewandavies.org}}.}
	\and
	Ross J. Kang%
	\thanks{Korteweg--de Vries Institute for Mathematics, University of Amsterdam, Netherlands. Supported by a Vidi grant (639.032.614) of NWO.
		Email: \protect\href{mailto:r.kang@uva.nl}{\protect\nolinkurl{r.kang@uva.nl}}.}
}

\date{\today}

\maketitle

\begin{abstract}
	List colouring is an influential and classic topic in graph theory. We initiate the study of a natural strengthening of this problem, where instead of one list-colouring, we seek many in parallel. Our explorations have uncovered a potentially rich seam of interesting problems spanning chromatic graph theory.

Given a $k$-list-assignment $L$ of a graph $G$, which is the assignment of a list $L(v)$ of $k$ colours to each vertex $v\in V(G)$, we study the existence of $k$ pairwise-disjoint proper colourings of $G$ using colours from these lists. We may refer to this as a \emph{list-packing}. Using a mix of combinatorial and probabilistic methods, we set out some basic upper bounds on the smallest $k$ for which such a list-packing is always guaranteed, in terms of the number of vertices, the degeneracy, the maximum degree, or the (list) chromatic number of $G$. (The reader might already find it interesting that such a minimal $k$ is well defined.) We also pursue a more focused study of the case when $G$ is a bipartite graph. Our results do not yet rule out the tantalising prospect that the minimal $k$ above is not too much larger than the list chromatic number.

	Our study has taken inspiration from study of the strong chromatic number, and we also explore generalisations of the problem above in the same spirit.

\end{abstract}

\section{Introduction}

Some resource allocation problems may be framed in terms of graph colouring, for example if resources are colours to be allocated to vertices of a graph such that vertices which cannot make simultaneous use of any resource are connected by edges.
In such a situation, conceivably one might desire not just one colouring, but several in parallel. One might even want that collectively they cover all possible resource usage. What general conditions guarantee this? It is this strengthened graph colouring problem that we systematically study in this work.

A particularly natural way to frame this is with respect to \emph{list colouring}~\cite{Viz76,ERT80}. Given a graph $G$, a \emph{list-assignment} $L$ of $G$ is a mapping such that $L(v)$ is a subset of natural numbers (list of colours) associated to the vertex $v\in V(G)$.
By considering the colours as resources, one can interpret each list as reflecting the specific availability of resources at some site based on some external restrictions.
Given a positive integer $k$, a \emph{$k$-list-assignment} is a list-assignment for which each list has cardinality $k$.
A \emph{proper $L$-colouring} is a mapping $c: V(G)\to {\mathbb N}$ such that $c(v)\in L(v)$ for any $v\in V(G)$ and such that whenever $uv$ is an edge of $G$, $c(u)\ne c(v)$.
That is, $c$ encodes a proper colouring of the vertices of $G$ such that each vertex is coloured by a colour of its list.
The \emph{list chromatic number} (or \emph{choosability}) $\chi_\ell(G)$ is the least $k$ such that $G$ admits a proper $L$-colouring for \emph{any} $k$-list-assignment $L$ of $G$.
Note that $\chi_\ell(G)$ is always at least the chromatic number $\chi(G)$, as one of the list-assignments one must consider gives the same list to every vertex.

We formulate the main question above concretely within the framework of list colouring. Given a list-assignment $L$ of $G$, an \emph{$L$-packing of $G$ of size $k$} is a collection of $k$ mutually disjoint $L$-colourings $c_1,\dots,c_k$ of $G$, that is, $c_i(v)\ne c_j(v)$ for any $i\ne j$ and any $v\in V(G)$. We say that an $L$-packing is proper if each of the disjoint $L$-colourings is proper.
We define the \emph{list (chromatic) packing number} $\chi^\star_\ell(G)$ of $G$ as the least $k$ such that $G$ admits a proper $L$-packing of size $k$ for any $k$-list-assignment $L$ of $G$.
Note that $\chi^\star_\ell(G)$ is necessarily at least $\chi_\ell(G)$.
The latter implies monotonicity: for every $k>\chi^\star_\ell(G)$ and any $k$-list-assignment $L$ of $G$, iteratively one can find $k-\chi^\star_\ell(G)$ disjoint $L$-colourings and finally a proper $L$-packing of $G$ by adding $\chi^\star_\ell(G)$ disjoint $L$-colourings.

It might not be immediately obvious that the parameter $\chi^\star_\ell$ is always well-defined, but we provide a number of different proofs of this fact in the course of our work.
It was Alon, Fellows and Hare~\cite{AFH96} who first suggested the study of $\chi^\star_\ell$ right at the end of their paper, but ours is the first work to embrace this suggestion.

Upon encountering the definition of list packing number, a natural course of action is to pursue list packing analogues of the most basic and important results on list colouring. This indeed is our programme, and we present several results in this vein. Our programme would be partly redundant if we were able to prove the following conjecture.
\begin{conjecture}\label{conj:main}
	There exists $C>0$ such that $\chi^\star_\ell(G) \le C \cdot\chi_\ell(G)$ for any graph $G$.
\end{conjecture}
\noindent
It could even be the case, moreover, that, for each $\eps>0$ there is some $\chi_0$ such that $\chi^\star_\ell(G) \le (1+\eps) \cdot\chi_\ell(G)$ for all $G$ with $\chi_\ell(G)\ge \chi_0$.
A resolution of Conjecture~\ref{conj:main} or its asymptotically stronger variant, either affirmatively or negatively, would be very interesting.

So far we only know that $\chi^\star_\ell(G)$ is upper bounded by some exponential function of $\chi_\ell(G)$; see Theorem~\ref{thm:degeneracy} below.
Let us observe that $\chi^\star_\ell(G)$ can indeed be strictly larger than $\chi_\ell(G)$ for some $G$. This is the case when $G$ is an even cycle. For example, consider the cycle of length $4$ and lists $\{1,2\}$, $\{1,2\}$, $\{1,3\}$, $\{2,3\}$, listed in cyclic order.
Incidentally, this implies that the constant $C$ in Conjecture~\ref{conj:main}, if it exists, may not be smaller than $3/2$.

Our main contribution is to present a collection of bounds on $\chi^\star_\ell$ within a few basic graph colouring settings.
Please consult Subsection~\ref{sub:notation} for definitions of any unfamiliar or ambiguous graph theoretic terminology.
Due to the depth and breadth of chromatic graph theory, there remains a plethora of further possibilities. We summarise what we consider to be the most tempting ones at the end of the paper.

Perhaps the simplest upper bound on the chromatic number of a graph is the number of vertices, and this bound easily carries over to list chromatic number too.
In Section~\ref{sec:order}, we show using an unexpectedly delicate inductive argument that the analogous statement for list packing number also holds.

\begin{theorem}\label{thm:n}
	$\chi^\star_\ell(G) \le n$ for any graph $G$ on $n$ vertices. Equality holds if and only if $G$ is $K_n$, the complete graph on $n$ vertices.
\end{theorem}

Usually, one considers the bound $\chi_\ell(G)\le n$ as a corollary of a more refined statement about greedy colouring.
The degeneracy $\delta^\star(G)$ of a graph $G$ is the least $k$ such that there is an ordering of the vertices of $G$ such that each vertex has at most $k$ neighbours preceding it in the order.
A greedy algorithm that colours vertices in this order cannot stall if each list has size $1+\delta^\star(G)$, as at the time of colouring a vertex $v$ at most $\delta^\star(G)$ colours have been used previously on neighbours of $v$.
The degeneracy of a graph $G$ is at most the maximum degree $\Delta(G)$, and this is at most one less than the number of vertices, giving the well-known bounds $\chi_{\ell}(G) \le 1 + \delta^\star(G) \le 1 + \Delta(G) \le n$ for an $n$-vertex graph $G$.
In Sections~\ref{sec:degeneracy} and~\ref{sec:maxdegree}, using what are essentially greedy arguments, we show the following two upper bounds on list packing number in terms of the degeneracy and maximum degree, respectively.

\begin{theorem}\label{thm:degeneracy}
	$\chi^\star_\ell(G) \le 2\delta^\star(G)$ for any graph $G$.
\end{theorem}

\begin{theorem}\label{thm:maxdegree}
	$\chi^\star_\ell(G) \le 1+\Delta(G) + \chi_\ell(G)$ for any graph $G$.
\end{theorem}

\noindent
Although neither implies Theorem~\ref{thm:n}, these two results have a number of interesting consequences. Both imply that $\chi^\star_\ell(G)$ is always at most around $2\Delta(G)$, a bound which is sharp up to the factor $2$.  They imply that for Conjecture~\ref{conj:main}, one may restrict attention to graphs having list chromatic number much smaller than the degeneracy or maximum degree. Together with the result proved by Alon~\cite{Alo93} that for some constant $C>0$ $\chi_\ell(G) \ge C\log \delta^\star(G)$ for all $G$, note that Theorem~\ref{thm:degeneracy} implies that $\chi^\star_\ell(G)$ is bounded by an exponential function of $\chi_\ell(G)$, which serves as modest support for Conjecture~\ref{conj:main}. Theorem~\ref{thm:degeneracy} together with Euler's formula further implies that $\chi^\star_\ell(G)\le 10$ for any planar graph $G$. It is worth noting that the bound in Theorem~\ref{thm:degeneracy} cannot be improved to $1+\delta^\star(G)$, see Theorem~\ref{thm:degeneracyexample}.

Another important way to bound the list chromatic number is to estimate it in terms of the chromatic number. A basic but well-known result of Alon~\cite{Alo92} asserts that $\chi_\ell(G)$ is within a factor $\log n$ of the chromatic number $\chi(G)$, where $n$ denotes the order of the graph $G$. We show the following two list packing versions of this in Section~\ref{sec:approx}. Here $\chi_f(G)$ denotes the fractional chromatic number of $G$, while $\rho(G)$ denotes the Hall ratio of $G$.

\begin{theorem}\label{thm:fractional}
	$\chi^\star_\ell(G) \le  \frac{(5+o(1)) \log n}{\log(\chi_f(G)/(\chi_f(G)-1))}\le (5+o(1))\chi_f(G)\log n$ for any graph $G$ with at least one edge and with $n$ vertices, as $n\to\infty$.
	If $\chi_f(G)$ is bounded as $n\to\infty$ then the leading constant can be improved to $1$.
\end{theorem}

\begin{theorem}\label{thm:hallratio}
	$\chi^\star_\ell(G) \le (5+o(1))\rho(G) (\log n)^2$ for any graph $G$ on $n$ vertices, as $n\to\infty$.
\end{theorem}

\noindent
Our proofs of these results rely on good estimates for the probability of zero permanent in a random binary matrix (Theorem~\ref{thr:Permenant_BernoulliDistributedMatrix}).
In the case that $\chi_f(G)$ is bounded, Theorem~\ref{thm:fractional} is asymptotically sharp for the complete multipartite graphs (see~\cite{ERT80}).
And since $\chi_f(G)\le \chi(G)$ and since $1/\log(x/(x-1)) \le x$, Theorem~\ref{thm:fractional} is a stronger form of the last-mentioned result of Alon.
Since $\chi_\ell(G)\le\chi^\star_\ell(G)$, Theorem~\ref{thm:fractional} strengthens a recent result in~\cite{CJKP20}, while Theorem~\ref{thm:hallratio} partially strengthens another recent result in~\cite{NoPo23}. It could well be that the extra factor $\log n$ in Theorem~\ref{thm:hallratio} is unnecessary, in which case, since $\rho(G)\le\chi_f(G)$ and since $1/\log(x/(x-1)) \sim x$ as $x\to\infty$, we would then essentially have a common strengthening of Theorems~\ref{thm:fractional} and~\ref{thm:hallratio}.

Note that in the special case of a bipartite graph $G$ on $n$ vertices, Theorem~\ref{thm:fractional} (or see also Theorem~\ref{th:bipartiteLogN} below) implies that $\chi_\ell(G) \le (1+o(1))\log_2 n$ as $n\to\infty$. Alon and Krivelevich~\cite{AlKr98} conjectured something more refined in terms of maximum degree, in particular, that for some $C>0$ we have $\chi_\ell(G) \le C\log \Delta(G)$ for any bipartite $G$. Since its formulation there has been surprisingly little progress on this essential problem. Already then, it was known that for some $C>0$ $\chi_\ell(G) \le C\Delta/\log \Delta(G)$ for any bipartite $G$, a statement which is a corollary of the seminal result of Johansson for triangle-free graphs~\cite{Joh96+}. Recent related efforts have only affected the asymptotic leading constant $C$, bringing it down to $1$; see~\cite{Mol19,ACK21,CaKa22}. Our work matches these recent efforts, but for qualitatively a much stronger structural parameter.

\begin{theorem}\label{thm:bipartite}
	$\chi^\star_\ell(G) \le (1+o(1)) \Delta/\log \Delta$ for any bipartite graph $G$ with $\Delta(G)\le\Delta$, as $\Delta\to\infty$.
\end{theorem}

\noindent
Combining Conjecture~\ref{conj:main} and the Alon--Krivelevich conjecture mentioned above, one could already aspire to a much better bound. (Or to put it in another way, the bipartite case might be where it is most natural to seek counterexamples to Conjecture~\ref{conj:main}.)
On the other hand, meeting the `Shearer barrier'~\cite{She83} here is difficult enough.
Our proof of Theorem~\ref{thm:bipartite} is rather involved. We draw on the `coupon collector intuition', which we know can be neatly combined with the Lov\'asz local lemma to obtain the analogous list chromatic number bound (see~\cite{ACK21,CaKa22}). Unfortunately, this approach does not immediately apply here for Theorem~\ref{thm:bipartite} as it is impossible to extend the requisite negative correlation property for packing; however, we managed to circumvent this obstacle for a suitable result on transversals in random matrices (Lemma~\ref{cor:low_probability_zerotransversal}) by using significant further probability estimates.
It seems challenging to establish the bound of Theorem~\ref{thm:bipartite} for all triangle-free $G$, even at the expense of the asymptotic leading constant. Already a bound strictly smaller than $\Delta$ would be welcome progress.

Although we have presented our results so far purely in terms of the list packing number, many of them can be shown in some natural stronger or more general settings. If one interprets list-colourings as independent transversals of some suitably vertex-partitioned auxiliary graph, then one can similarly reinterpret list-packings. By considering generalisations of that auxiliary graph, we obtain a natural range of parameters/settings that expands from list packing towards the classic notion of strong colourings~\cite{Alo88,Fel90}.
We introduce and discuss these in Subsection~\ref{sub:moregeneral}. Thereafter we `resume' the introduction to this paper, by also stating and discussing the strengthened or more general versions of our results. Before that, we set out some of our more basic notation, for clarity and completeness.

\subsection{Basic terminology}\label{sub:notation}
For the convenience of the reader, here we indicate some of the standard notational conventions/definitions we assume throughout this work.
Let $G$ be a graph. The \emph{order} of $G$ is $|G|=|V(G)|$.
For $v\in V(G)$, the \emph{neighbourhood} of $v$ is $N(v) = \{u\in V(G):uv\in E\}$ and the \emph{degree} of $v$ is $\deg(v)=|N(v)|$.

The \emph{minimum degree} of $G$ is $\delta(G)=\min\{\deg(v):v\in V(G)\}$ and the \emph{maximum degree} of $G$ is $\Delta(G)=\max\{\deg(v):v\in V(G)\}$. The \emph{degeneracy} of $G$ is $\delta^\star(G)=\max_{H\subseteq G}\delta(H)$, where the maximum is taken over all subgraphs $H$ of $G$, and this is equivalent to the definition given above in terms of neighbours preceding vertices in some order.
A \emph{clique} in a graph $G$ is a set of vertices $X\subseteq V(G)$ such that every pair in $X$ forms an edge of $G$, and an \emph{independent set} in $G$ is a set $Y$ such that no pair in $Y$ forms an edge of $G$.
The \emph{clique number} $\omega(G)$ of $G$ is size of a largest clique in $G$ and the \emph{independence number} $\alpha(G)$ of $G$ is size of a largest independent set in $G$.

The \emph{Hall ratio} of $G$ is $\rho(G) = \max_{H\subseteq G}|H|/\alpha(H)$, where the maximum is taken over subgraphs $H$ of $G$ containing at least one vertex.
The \emph{fractional chromatic number $\chi_f(G)$} of $G$ is the least $k$ such that there exists a probability distribution over the independent sets of $G$ such that $\Pr(v\in \mathbf{S}) \ge 1/k$ for each $v\in V(G)$ and $\mathbf{S}$ a random independent set drawn according to the distribution.
The \emph{chromatic number $\chi(G)$} of $G$ is the least $k$ such that there is a partition of $V(G)$ into $k$ independent sets.
We will usually use $n$ to denote the number of vertices of a graph $G$.

It is worth keeping in mind that the following simple inequalities always hold:
\begin{align*}
	\omega(G)\le\rho(G)\le\chi_f(G)\le\chi(G)\le\chi_\ell(G)\le1+\delta^\star(G)\le1+\Delta(G) \le n.
\end{align*}

Given a matrix $A$, we denote its entries with $A_{i,j}$ and with $A_{I \times J}$ we denote the submatrix formed from rows $I$ and columns $J$.
A \emph{binary matrix} is a matrix with entries in $\{0,1\}$. A \emph{zero matrix} is a matrix with every entry equal to $0$.
A binary $k \times k$ matrix is said to have a \emph{$1$-transversal} (resp.\ \emph{$0$-transversal}) if it contains $k$ elements  that are each of value $1$ (resp.\ each of value $0$), such that no two elements are in the same column and no two are in the same row.
The \emph{permanent} of a $k\times k$ matrix $A$ is defined as
$\Per (A)=\sum_{\sigma}\prod_{i=1}^{n}A_{i,\sigma(i)},$ where the sum ranges over all permutations $\sigma$ of $[k]=\left\{1,\ldots,k\right\}$.
For our applications, it is important to note that a binary matrix $A$ has no $1$-transversal if and only if $\Per(A)=0$.

\subsection{From list packing to strong colouring}\label{sub:moregeneral}

As mentioned, although we have chosen to introduce our results in terms of the parameter $\chi^\star_\ell$, many of our results can be established (and otherwise pursued) in stronger or more general settings. We define and discuss these now.

Let $G$ and $H$ be graphs.
We say that $H$ is a \emph{cover} of $G$ with respect to a mapping $L: V(G) \to 2^{V(H)}$ if
$L$ induces a partition of $V(H)$ and the subgraph induced between $L(v)$ and $L(v')$ is empty whenever $vv'\notin E(G)$.
Under various basic conditions on $G$, $H$ and $L$, we are interested in the existence of an \emph{independent transversal} of $H$ with respect to the partition, which is a vertex subset with exactly one vertex chosen from each part that simultaneously forms an independent set.
Moreover, we say that a cover $H$ of $G$ with respect to $L$ is \emph{$k$-fold} if $|L(v)|=k$ for all $v\in V(G)$. In this case, we will also be interested in the existence of some \emph{($k$-fold) IT packing} of $H$ with respect to $L$, i.e.~a collection of $k$ mutually disjoint independent transversals of $H$ with respect to $L$.
These settings are motivated in a light-hearted way in a well-known survey paper of Haxell~\cite{Hax11}.

To illustrate, we first reformulate list colouring and list packing in this terminology.
From a list-assignment $L$ of $G$, we derive the \emph{list-cover} $H_\ell(G,L)$ for $G$ via $L$ as follows.
For every $v\in V(G)$, we let $L_\ell(v)=\{(v,c)\}_{c\in L(v)}$ and define $V(H_\ell)=\cup_{v\in V(G)}L_\ell(v)$. We define $E(H_\ell)$ by including $(v,c)(v',c')\in E(H_\ell)$ if and only if $vv'\in E(G)$ and $c=c'\in L(v)\cap L(v')$.
Note that $H_\ell$ is a cover of $G$ with respect to $L_\ell$, and that $H_\ell$ is $k$-fold if $L$ is a $k$-list-assignment of $G$.
The independent transversals of $H_\ell$ with respect to $L_\ell$ are in one-to-one correspondence with the proper $L$-colourings of $G$.
If $L$ is a $k$-list-assignment of $G$, then the $k$-fold IT packings of $H_\ell$ with respect to $L_\ell$ are in one-to-one correspondence with the proper $L$-packings of $G$ of size $k$.
Thus $\chi_\ell(G)$ and $\chi^\star_\ell(G)$ can equivalently be defined in terms of list-covers.

We can somewhat relax the setting as follows.
A \emph{correspondence-assignment} for $G$ via $L$ is a cover $H$ for $G$ via $L$ such that for each edge $vv'\in E(G)$, the subgraph induced between $L(v)$ and $L(v')$ is a matching. We call such an $H$ a \emph{correspondence-cover} for $G$ with respect to $L$.
A \emph{correspondence $L$-colouring} of $H$ is an independent transversal of $H$ with respect to $L$.
The \emph{correspondence chromatic number} (or \emph{DP-chromatic number}) $\chi_c(G)$ is the least $k$ such that any $k$-fold correspondence-cover $H$ of $G$ via $L$ admits a correspondence $L$-colouring~\cite{DvPo18}.
Given a $k$-fold correspondence-cover, a \emph{($k$-fold) correspondence $L$-packing of $H$} is a $k$-fold IT packing of $H$ with respect to $L$.
We define the \emph{correspondence (chromatic) packing number} $\chi^\star_c(G)$ of $G$ as the least $k$ such that any $k$-fold correspondence-cover $H$ of $G$ via $L$ admits a $k$-fold correspondence $L$-packing.
Every list-cover is a correspondence-cover, and moreover, for any list-cover $H$ of $G$ via $L$, the correspondence $L$-colourings of $H$ are in one-to-one correspondence with the proper $L$-colourings of $G$. In this sense correspondence colouring (or packing) is a relaxation of list colouring (or packing, respectively).
Note that $\chi_c(G)$ is necessarily at least $\chi_\ell(G)$,  $\chi^\star_c(G)$ is necessarily at least $\chi^\star_\ell(G)$, and $\chi^\star_c(G)$ is necessarily at least $\chi_c(G)$.
It is easily observed that $\chi_c(G) \le 1+\delta^\star(G)$ always.

In the study of $\chi_\ell(G)$ and $\chi_c(G)$ (and indeed also of $\chi^\star_\ell(G)$ and $\chi^\star_c(G)$), it is natural to explicitly take into account structural properties of the graph $G$, and this in fact is a large and ever-expanding subfield of graph theory; however, given how we may define the parameters as in this subsection, it is also valid and interesting to incorporate the properties of the cover graph $H$ instead, without any explicit assumptions on $G$.
With respect to list-covers, correspondence-covers, or general covers, the existence of independent transversals in these terms has already attracted quite some interest, particularly under the condition that the cover graph has bounded maximum degree.

For general covers, such a problem was investigated by Bollob\'as, Erd\H{o}s and Szemer\'edi~\cite{BES75} in the mid 1970s.
For each $D$, they essentially asked for the least integer $\Lambda(D)$ for which the following holds: for any $\Lambda(D)$-fold cover $H$ of some graph $G$ via $L$ such that $\Delta(H) \le D$, there is an independent transversal of $H$ with respect to $L$. Using topological methods, Haxell~\cite{Hax01} proved that $\Lambda(D) \le 2D$, a result which is sharp for every $D$~\cite{SzTa06}.

For list-covers, such a problem was proposed by Reed~\cite{Ree99}. For each $D$, he essentially asked for the least integer $\Lambda_\ell(D)$ for which the following holds: for any $\Lambda_\ell(D)$-fold list-cover $H$ of some graph $G$ via $L$ such that $\Delta(H) \le D$, there is an independent transversal of $H$ with respect to $L$. In fact, he conjectured that $\Lambda_\ell(D)=D+1$. This was disproved by Bohman and Holzman~\cite{BoHo02}; however, with nibble methodology Reed and Sudakov~\cite{ReSu02} proved that $\Lambda_\ell(D) \le D+o(D)$ as $D\to\infty$.

For correspondence-covers, such a problem was proposed by Aharoni and Holzman, see~\cite{LoSu07}. For each $D$, they essentially asked for the least integer $\Lambda_c(D)$ for which the following holds: for any $\Lambda_c(D)$-fold correspondence-cover $H$ of some graph $G$ via $L$ such that $\Delta(H) \le D$, there is an independent transversal of $H$ with respect to $L$. Again with nibble methods, Loh and Sudakov~\cite{LoSu07} proved that $\Lambda_c(D) \le D+o(D)$ as $D\to\infty$. We note that very recently two independent works~\cite{GS22,KaKe22} showed a substantial strengthening of this last result in terms of a part-averaged degree condition for the correspondence-cover.

For each of the three parameters $\Lambda$, $\Lambda_\ell$, $\Lambda_c$, one can define the analogous packing variants. That is, for each $D$, what is the least integer $\Lambda^\star(D)$ (or $\Lambda^\star_\ell(D)$ or $\Lambda^\star_c(D)$) for which the following holds: for any $\Lambda^\star(D)$-fold cover (or $\Lambda^\star_\ell(D)$-fold list-cover or $\Lambda^\star_c(D)$-fold correspondence-cover, respectively) $H$ of some graph $G$ via $L$ such that $\Delta(H) \le D$, there is an IT packing of $H$ with respect to $L$.
Indeed, under the guise of \emph{strong colouring}, $\Lambda^\star(D)$ has already (long) been investigated: independently, Alon~\cite{Alo88} and Fellows~\cite{Fel90} first considered this parameter, and in particular Alon showed using the L\'ovasz local lemma that $\Lambda^\star(D)=O(D)$. Haxell~\cite{Hax08} proved that $\Lambda^\star(D)\le 2.75D+o(D)$ as $D\to\infty$, which is the best known bound to date. The folkloric \emph{strong colouring conjecture} (see~\cite{ABZ07}) asserts that $\Lambda^\star(D)\le 2D$ for all $D$.
Note that this longstanding conjecture has a hint of (and indeed inspired) our Conjecture~\ref{conj:main}.

\subsection{Introduction continued}\label{sub:continued}

In Subsection~\ref{sub:moregeneral}, we have presented essentially three additional problem settings `between' list packing and strong colouring. Equipped with these notions, we can further discuss our results and trajectory in this context.

Let us resume the introduction by first considering how the results on the list packing number extend (or not) to the correspondence packing number, and contrast the respective behaviours. Note that, in this realm, we can maintain our view on the `covered' graph $G$, and so the panoply of results and problems in graph colouring can still guide our explorations for correspondence packing. We note that, just as before, these explorations would be made partly redundant if we were able to prove the following conjecture.

\begin{conjecture}\label{conj:maincorr}
	There exists $C>0$ such that $\chi^\star_c(G) \le C \cdot\chi_c(G)$ for any graph $G$.
\end{conjecture}
\noindent
We will see later for this assertion both that it is true up to a logarithmic factor and that it cannot be true with $C$ taken smaller than $2$.

In terms of the number of vertices, the problem of correspondence packing has recently been highlighted (in different terminology) by Yuster~\cite{Yus21}. Casting it as a specialisation of earlier conjectures related to the Hajnal--Szemer\'edi theorem, he stated the conjecture, rephrased in our terminology, that $\chi^\star_c(K_n)=n$ if $n$ is even and $\chi^\star_c(K_n) = n+1$ if $n$ is odd. This assertion would be best possible if true, due to a construction of Catlin~\cite{Cat80}. Note that Catlin's construction also (barely) precludes a direct correspondence packing analogue of Theorem~\ref{thm:n}. Yuster used semirandom methods to prove that $\chi_c^\star(K_n)\le 1.78 n$ for all sufficiently large $n$.

In terms of degree parameters, we indeed have the following generalisations of Theorems~\ref{thm:degeneracy} and~\ref{thm:maxdegree} for correspondence packing.

\begin{theorem}\label{thm:degeneracycorr}
	$\chi^\star_c(G) \le 2\delta^\star(G)$ for any graph $G$.
\end{theorem}

\begin{theorem}\label{thm:maxdegreecorr}
	$\chi^\star_c(G) \le 1+\Delta(G) + \chi_c(G)$ for any graph $G$.
\end{theorem}

\noindent
Note first that since $\chi^\star_\ell(G)\le \chi^\star_c(G)$, Theorem~\ref{thm:degeneracycorr} implies Theorem~\ref{thm:degeneracy}. Theorem~\ref{thm:maxdegreecorr} does not immediately imply Theorem~\ref{thm:maxdegree}, but their proofs are almost identical.
Again, both results imply that $\chi^\star_c(G)$ is always at most around $2\Delta(G)$, a bound which is sharp up to the factor $2$.  Together with a result noticed independently by Bernshteyn~\cite{Ber16} and Kr\'{a}\v{l}, Pangr\'{a}c and Voss~\cite{KPV05} that for some constant $C>0$ $\chi_c(G) \ge C\delta^\star(G)/\log \delta^\star(G)$ for all $G$, note that Theorem~\ref{thm:degeneracycorr} implies that $\chi^\star_c(G)$ is at most a logarithmic factor larger than $\chi_c(G)$, in support of Conjecture~\ref{conj:maincorr}. Theorem~\ref{thm:degeneracycorr} together with Euler's formula further implies that $\chi^\star_c(G)\le 10$ for any planar graph $G$. In Section~\ref{sec:degeneracy}, we exhibit an example (see Proposition~\ref{prop:degeneracyDP} below) to show that the bound in Theorem~\ref{thm:degeneracycorr} is best possible. Since $\chi_c(G) \le 1+\delta^\star(G)$ always, this example also shows that $C$ in Conjecture~\ref{conj:maincorr} (if true) cannot be taken smaller than $2$.

In terms of direct extensions of Theorems~\ref{thm:fractional} and~\ref{thm:hallratio} (in terms of fractional chromatic number and Hall ratio) to correspondence packing, the just-mentioned result in~\cite{Ber16} and~\cite{KPV05} proves it impossible, as the corresponding statements already fail for complete bipartite graphs.
To be sure, complete bipartite graphs are where the difference between list packing and correspondence packing is clearest.
We are able to establish the correspondence packing strengthening of Theorem~\ref{thm:bipartite}.

\begin{theorem}\label{thm:bipartitecorr}
	$\chi^\star_c(G) \le (1+o(1)) \Delta/\log \Delta$ for any bipartite graph $G$ with $\Delta(G)\le\Delta$, as $\Delta\to\infty$.
\end{theorem}

\noindent
However, the complete bipartite graph $K_{\Delta,\Delta}$ certifies this statement to be sharp up to an asymptotic multiple of $2$, due to the observation from~\cite{Ber16} and~\cite{KPV05}. This gives a sharp contrast with Theorem~\ref{thm:fractional}.  This contrast might appear to follow mainly from the known difference in behaviour between list and correspondence chromatic numbers. On the other hand, we also demonstrate that there are complete bipartite graphs having list chromatic, list packing and correspondence chromatic numbers all equal, but correspondence packing number almost twice as large (Corollary~\ref{cor:bipartite}).

Just as for list packing, we suspect that a more general result holds, namely that Theorem~\ref{thm:bipartitecorr} could be extended to all triangle-free $G$.

\begin{conjecture}\label{conj:trianglefree}
	$\chi^\star_c(G) \le (1+o(1)) \Delta/\log \Delta$ for any triangle-free graph $G$ with $\Delta(G)\le\Delta$, as $\Delta\to\infty$.
\end{conjecture}
\noindent
If true, this conjecture would be a striking finding following in a long and acclaimed sequence of results, including the work of Johansson~\cite{Joh96+} and Molloy~\cite{Mol19}, among others.

As we alluded to in Subsection~\ref{sub:moregeneral}, our trajectory tends naturally towards a couple of problem settings interpolating between list/correspondence packing and strong colouring.
For example, analogously to the strong colouring conjecture, it would be interesting to investigate the following possibilities for the two packing parameters introduced at the end of Subsection~\ref{sub:moregeneral}.
\begin{conjecture}\label{conj:ITpacking}
	$\Lambda^\star_\ell(D) \le D+o(D)$ and $\Lambda^\star_c(D) \le D+o(D)$ as $D\to\infty$.
\end{conjecture}
\noindent
Note that this conjecture implies the following conjectural statements for list and correspondence packing number: for each $\eps>0$ there is some $\Delta_0$ such that $\chi^\star_\ell(G)\le(1+\eps)\Delta$ (respectively, $\chi^\star_c(G)\le(1+\eps)\Delta$) for all $G$ with $\Delta(G)\ge \Delta_0$.
We also note that the argument outlined in Section~\ref{sec:maxdegree} for Theorems~\ref{thm:maxdegree} and~\ref{thm:maxdegreecorr} also shows that $\Lambda^\star_\ell(D)\le \Lambda^\star_c(D)\le 2D+o(D)$ as $D\to\infty$; see both Section~\ref{sec:maxdegree} and a remark in~\cite{LoSu07}.

In considering problems such as Conjecture~\ref{conj:ITpacking}, it is also sensible to explore what holds under some basic structural conditions on the cover graph.
Along these lines, we would like to emphasise the possibilities related to Theorems~\ref{thm:bipartite} and~\ref{thm:bipartitecorr} and Conjecture~\ref{conj:trianglefree} in particular.
Let us reformulate Theorem~\ref{thm:bipartitecorr}: there is some $k=k(\Delta)$ satisfying $k\le (1+o(1))\Delta/\log \Delta$ as $\Delta\to\infty$ such that for any bipartite graph $G$ with $\Delta(G)\le \Delta$, it holds that any $k$-fold correspondence-cover $H$ of $G$ via $L$ admits a correspondence $L$-packing. In a conservative manner, we could propose the following stronger statement.
\begin{conjecture}\label{conj:CaKa}
	There is some $k=k(\Delta)$ satisfying $k\le (1+o(1))\Delta/\log \Delta$ as $\Delta\to\infty$ such that the following holds.
	Suppose that $G$ and $H$ are bipartite graphs such that $H$ is a $k$-fold correspondence-cover of $G$ via some $L$, and that $\Delta(H)\le \Delta$. Then $H$ admits a correspondence $L$-packing.
\end{conjecture}
\noindent
If true, this would constitute a correspondence packing analogue of a recent result due to two of the authors~\cite{CaKa22}.
Beyond a common strengthening of Conjectures~\ref{conj:trianglefree} and~\ref{conj:CaKa}, one could more boldly speculate the following.
\begin{conjecture}\label{conj:bold}
	For every $r\ge 3$, there is some $C_r>0$ such that the following holds for $k= C_r\Delta/\log\Delta$.
	Suppose that $G$ and $H$ are graphs such that $H$ is a $k$-fold correspondence-cover of $G$ via some $L$, $H$ contains no copy of $K_r$, and $\Delta(H)\le \Delta$. Then $H$ admits a correspondence $L$-packing.
\end{conjecture}
\noindent
If true, this would imply a recent conjecture of Anderson, Bernshteyn and Dhawan~\cite{ABD23}, which in turn would imply a conjecture of Alon, Krivelevich and Sudakov~\cite{AKS99}, which in turn would imply a conjecture of Ajtai, Erd\H{o}s, Koml\'os and Szemer\'edi~\cite{AEKS81}. Due to the connections with these last two central and longstanding conjectures, even partial progress on Conjecture~\ref{conj:bold}, in either direction, would be intriguing.

\subsection{Outline of the paper}\label{sub:outline}

See Tables~\ref{table:overview} and~\ref{table:auxiliary} for an overview of our results that refers to where we state and prove them in the paper. In Section~\ref{sec:order} we prove Theorem~\ref{thm:n} by showing that $\chi_{\ell}^\star(K_n)=n$ and $\chi_{\ell}^\star(K^-_{n})=n-1$, where $K_n^-$ equals $K_n$ minus one edge.
In Section~\ref{sec:degeneracy} we prove Theorem~\ref{thm:degeneracy} by doing this for the correspondence version, Theorem~\ref{thm:degeneracycorr}, through an application of Hall's marriage theorem.
In Section~\ref{sec:maxdegree} we prove Theorem~\ref{thm:maxdegreecorr} by showing that a partial colouring that does not use colours from every list is not a maximum partial colouring. Theorem~\ref{thm:maxdegree} is proved with exactly the same method.
In Section~\ref{sec:randommatrix} we generalise a theorem of Everett and Stein~\cite{ES73} on the permanent of random $k\times k$ matrices. This essentially says that (under certain conditions) the probability that a random binary matrix $A$ has a permanent equal to zero is approximately the probability that it contains a row or column of zeros.
This is applied in Section~\ref{sec:approx} to prove Theorems~\ref{thm:fractional} and~\ref{thm:hallratio}. Here we construct some matrices for every $M^\star(v)$ for every vertex $v$ in which a $1$-transversal corresponds with an ordering of the colours in the list, resulting in different $L$-colourings.
In Section~\ref{sec:bip} we prove that the list packing number of a bipartite graph with parts having maximum degrees $\Delta_A$ and $\Delta_B$ is at most $\min\{\Delta_A, \Delta_B\}+1$.
We show sharpness of this result and prove that there are complete bipartite graphs for which the correspondence packing number is nearly twice the correspondence chromatic number.
At the end, we prove Theorem~\ref{thm:bipartitecorr}, and thus also Theorem~\ref{thm:bipartite}.

  \begin{table}[htb!]
  \centering
  \begin{tabular}{ |c || c | } 
    \hline 
    Result & Reference \\ 
    \hline \hline 
        $\chi_{\ell}^{\star}(G)\leq n $ & Thm.~\ref{thm:n}; Sec.~\ref{sec:order}  \\
    with equality if and only if $G$ is the complete graph $K_n$. & \\
    \hline
    $\chi_{c}^{\star}(G) \leq 2 d$  if $G$ is $d$-degenerate.&  Thm.~\ref{thm:degeneracycorr}, Prop.~\ref{prop:degeneracyDP} \\ 
    For all $d$, there is a $d$-degenerate $G$ with $\chi_{c}^{\star}(G) = 2d$. & Thm.~\ref{thm:degeneracyexample}; Sec.~\ref{sec:degeneracy} \\
  For all $d$, there is a $d$-degenerate $G$ with $\chi_{\ell}^{\star}(G) \geq d+2$. &\\
\hline
\ $\chi_{\ell}^{\star}(G) \leq 1+ \Delta + \chi_{\ell}(G)$ & Thms.~\ref{thm:maxdegree} and~\ref{thm:maxdegreecorr}; \\
\ $\chi_{c}^{\star}(G) \leq 1+ \Delta + \chi_c(G)$ & Sec.~\ref{sec:maxdegree} \\

\hline     \hline
$\chi_{\ell}^{\star}(G) \leq (5+o(1)) \cdot \chi_f(G) \cdot \log n$. & Thms.~\ref{thm:fractional},~\ref{th:bipartiteLogN},~\ref{thm:fractional_restated};\\
$\chi_{\ell}^{\star}(G) \leq (1+o(1)) \cdot \chi_f(G) \cdot \log n$, if $\chi_f(G)$ is bounded.  &  Sec.~\ref{sec:approx} \\
$\chi_{\ell}^{\star}(G) \leq (1+o(1)) \log_2 n$, if $G$ is bipartite. &\\    
    \hline   
    \hline
$\chi_{\ell}^{\star}(G) \leq \min(\Delta_A, \Delta_B)+1$ & Lem.~\ref{lem:bipgreedylist}; Sec.~\ref{sec:bip} \\    
\ if $G=(A\cup B,E)$ is bipartite with maximum degrees $\Delta_A, \Delta_B$. & \\
    \hline
\ $\chi_{\ell}^{\star}(K_{a,b})=b+1$  and $\chi_c^{\star}(K_{a,b})=2b$ & Cor.~\ref{cor:bipartite}; Sec.~\ref{sec:bip} \\
\ for all $b$ and $a=a(b)$ large enough.&\\
    \hline
    $\chi_{c}^{\star}(G) \leq (1+o(1)) \Delta /\log \Delta$  if $G$ is bipartite.& Thm.~\ref{thm:bipartitecorr}; Sec.~\ref{sec:bip} \\    
    \hline

  \end{tabular}
  \caption{An overview of our results on the list packing number $\chi_{\ell}^{\star}(G)$ and the correspondence packing number $\chi_c^{\star}(G)$ of a graph $G$. Here $n $ denotes the number of vertices and $\Delta$ is the maximum degree of $G$. 
Recall that $\chi_{\ell }^{\star}(G) \leq \chi_c^{\star}(G)$ for every graph $G$.
  }
  \label{table:overview}
  \end{table}

  \begin{table}[htb!]
  \centering
  \begin{tabular}{ |c || c | } 
    \hline 
    Result & Reference \\ 
    \hline \hline 
 A random $k\times k$ binary matrix with independent& \\
      Bernoulli($1-p$) distributed entries has no $1$-transversal & Thm.~\ref{thr:Permenant_BernoulliDistributedMatrix}; Sec.~\ref{sec:randommatrix} \\
     with probability $2k p^k (1+o(1))$,  as $k \rightarrow \infty$,&\\
    
 provided $k^5 p^k \rightarrow 0$. &\\
    \hline

        If $k\geq \lceil (1+\epsilon) n/ \log n \rceil$, then a sum of $n$ independent uniformly & \\
     random $k \times k$ permutation matrices has a $0$-transversal& Lem.~\ref{cor:low_probability_zerotransversal}; Sec.~\ref{sec:bip} \\
    with probability at least $1-3k^2 \cdot \exp(-n^{\epsilon/3})$.& \\
    \hline

  \end{tabular}
  \caption{Two of our auxiliary results on random matrices, perhaps of independent interest.}
  \label{table:auxiliary}
  \end{table}

\subsection{Further preliminaries}\label{sub:prelim}

In this subsection, we collect some basic results we require for the proofs.
We will use the following formulation of Hall's marriage theorem.

\begin{theorem}[\cite{Hal35}]
	Given a family $\mathcal{F}$ of finite subsets of some ground set $X$, where the subsets are counted with multiplicity, suppose $\mathcal{F}$ satisfies the \emph{marriage condition}, that is that for each subfamily $\mathcal{F}' \subseteq \mathcal{F}$
	\begin{align*}
		|\mathcal{F}'| \le \left|\bigcup_{A\in \mathcal{F}'} A\right|.
	\end{align*}
	Then there is an injective function $f: \mathcal{F}\to X$ such that $f(A)$ is an element of the set $A$ for every $A\in \mathcal{F}$, that is, the image $f(\mathcal{F})$ is a \emph{system of distinct representatives} of $\mathcal{F}$.
\end{theorem}

The following elementary criterion, also known as the Frobenius--K\"onig theorem, allows us to efficiently count the matrices that do not have a $1$-transversal. It is a consequence of Hall's theorem, but we include a proof for completeness.

\begin{lemma}[see~\cite{ASurveyOfMatrixTheoryandMatrixInequalitiesAllyandBacon1964}]\label{le:distinctrepr}
	Let $A=(A_{i,j}), i \in I,j\in J$ be a $k\times k$ binary matrix. Then $A$ has no $1$-transversal if and only if there exist $S \subseteq I$ and $T \subseteq J$ with $|S|+|T|>k$ such that the submatrix $A_{S\times T}$  is a zero matrix.
\end{lemma}

\begin{proof}
	For each column $j\in J$, let $R_j \subseteq I$ denote the set of rows $i$ such that $A_{ij}=1$.
	Observe that $A$ has no $1$-transversal if and only if $\{R_1,\ldots, R_{|J|}\}$ has no system of distinct representatives. Furthermore, by Hall's marriage theorem, this is true if and only if there exists $T \subseteq J$ such that $|\bigcup_{j \in T} R_j| < |T|$.

	Assuming the existence of such a set $T$, let $S:= I \setminus \bigcup_{j \in T} R_j$, which is the set of rows that do not contain any $1$ on any column of $T$. Then $A_{S\times T}$  is a zero matrix.  Furthermore, $|S|+|T|>|I|-|T|+|T|=k$, as desired. Conversely, if $A_{S\times T}$  is a zero submatrix with $|S|+|T|>k$, then $|\bigcup_{j \in T} R_j|\leq |I \setminus S|=k -|S| < |T|$, as required.
\end{proof}

Given a probability space, the $\{0,1\}$-valued random variables~$X_1,\dotsc,X_n$ are
\emph{negatively correlated} if for each subset~$S$ of~$\{1,\dotsc,n\}$,
\[
	\Pr\big(X_i=1, \forall i\in S\big)\le\prod_{i\in S}\Pr(X_i=1).
\]
Correspondingly, we say that a collection of events is negatively correlated if the random variables given by their indicator functions are negatively correlated. We will use the following Chernoff bound.

\begin{theorem}[see~\cite{PS97}]\label{thm:chernoff1}
	Let~$X_1, \dotsc, X_n$ be $\{0,1\}$-valued random variables, and let $X = \sum_{i=1}^{n}X_i$.
	If the variables~$X_1,\dotsc,X_n$ are negatively correlated, then
	\[
		\forall \delta \in [0,1],\quad \Pr\big(X\ge(1+\delta)\EE X\big) \le e^{-\delta^2\EE X/3}.
	\]
	If the variables~$1-X_1,\dotsc,1-X_n$ are negatively correlated, then
	\[
		\forall \eta\in [0,1],\quad \Pr\big(X\le(1-\eta)\EE X\big) \le e^{-\eta^2\EE X/2}.
	\]
\end{theorem}

\noindent
In our applications of Theorem~\ref{thm:chernoff1}, we will sometimes use the following simple lemma.
\begin{lemma}~\label{lem:negcorr}
	Let~$X_1, \dotsc, X_n$ be $\{0,1\}$-valued random variables such that $\Pr(X_i=X_j=1)=0$ for all distinct $i,j \in\{1,\dotsc,n\}$. Then $X_1, \ldots, X_n$ are negatively correlated and $1-X_1,\ldots, 1-X_n$ are negatively correlated.
\end{lemma}
\begin{proof}
	It is immediate that $X_1,\ldots,X_n$ are negatively correlated, so we only show that $1-X_1,\ldots, 1-X_n$ are negatively correlated.
	Let $S$ be a subset of $\{1, \dotsc, n \}$. Then for any subset $T\subseteq S$ of size at least two, $\EE\big( \prod_{i \in T } X_i \big)=\Pr\big( X_i=1, \forall i \in T\big)=0$. Therefore
	\begin{align*}
		\Pr \big(1-X_i=1, \forall i \in S\big) & =\EE\big[\prod_{i \in S} (1-X_i)\big]
		\\ &= 1- \sum_{i\in S} \EE X_i + \sum_{i <j \in S} \EE\big[X_iX_j\big] - \ldots + (-1)^{|S|} \cdot\EE \big[ \prod_{i\in S} X_i \big] \\&= 1- \sum_{i\in S} \EE X_i.
	\end{align*}
	As $\EE\big(X_i) \in [0,1]$ for all $i$, the right-hand side is a lower bound on $\prod_{i \in S}\big( 1-\EE X_i\big)$, which is easily proved by induction as in the simpler case of Bernoulli's inequality.
	We conclude that $1-X_1,\ldots, 1-X_n$ are negatively correlated because $1-\EE X_i = \Pr \big( 1-X_i=1\big)$.
\end{proof}

We will use the following symmetric form of the Lov\'asz local lemma due to Shearer.
\begin{theorem}[\cite{ErLo75,She85}]
	Consider a set $\mathcal{E}$ of (bad) events such that for each $A\in \mathcal{E}$
	\begin{enumerate}
		\item $\Pr(A) \le p < 1$, and
		\item $A$ is mutually independent of a set of all but at most $d$ of the other events.
	\end{enumerate}
	If $epd\le1$, then with positive probability none of the events in $\mathcal{E}$ occur.
\end{theorem}

%
%

\section{Order}\label{sec:order}

In this section we prove Theorem~\ref{thm:n}.
We first prove a lemma that implies $\chi_\ell^\star(K_n)=n$.

\begin{lemma}\label{lem:Kn}
	Given $1\le k \le n$, if $L$ is a $k$-list-assignment of $K_n$ such that every colour belongs to at most $k$ lists, then there is a proper $L$-packing.
\end{lemma}

\begin{proof}[Proof of Lemma~\ref{lem:Kn}]
	We proceed by induction on $k$. The statement is trivially true in the base case $k=1$.
	Assume $k>1$ and that the statement holds with $k$ replaced by $k-1$.

	Let $V$ be the vertex set of $K_n$.
	Let us call a colour \emph{rich} if it belongs to exactly $k$ lists, and write $R$ for the set of rich colours.
	If there is a proper $L$-colouring $c$ of $K_n$ that uses all the rich colours then we are done, as then the $(k-1)$-fold list-cover $L'$ such that $L'(v) = L(v)\setminus\{c(v)\}$ for all $v\in V$ has the property that every colour is used on at most $k-1$ lists.
	By induction, there is a proper $L'$-packing of $K_n$, which we may combine with $c$ to form a proper $L$-packing.
	So it suffices to construct the desired colouring $c$.

	We start by showing first that some proper $L$-colouring $c'$ of $K_n$ exists.
	Such a colouring is precisely a system of distinct representatives for the family $\mathcal{F} = \left\{ L(v) : v\in V\right\}$, and hence it suffices to show that the marriage condition holds. This entails showing that $|V'| \le \left|\bigcup_{v\in V'} L(v) \right|$ for every $V'\subseteq V$. This holds by a simple counting argument: every list has size $k$, and so counting colours with multiplicity on lists of $v\in V'$ gives precisely $k|V'|$. But every colour appears on at most $k$ lists by assumption, so appears with multiplicity at most $k$, and hence the total number of colours counted without multiplicity is at least $|V'|$, as required to guarantee the desired $c'$ by Hall's theorem.

	If $c'$ uses all the rich colours already then we are done, but if not we swap colours in $c'$ to ensure this property. First, note that there is a partial proper $L$-colouring which uses every rich colour by the following argument.
	For a rich colour $r\in R$, let $V_r = \{v\in V: r\in L(v)\}$, and let $\mathcal{F}' = \{V_r : r\in R\}$.
	Now every set in $\mathcal{F}'$ has size $k$ by the definition of $R$, and each $L(v)$ has size $k$. The marriage condition holds for $\mathcal{F}'$ by the same counting argument as above: given a subset $R'\subseteq R$, with multiplicity we count at least $k|R'|$ vertices whose lists contain colours in $R'$, but the multiplicities are at most the list size $k$.
	Hall's theorem thus guarantees that there is an injective function $f : R\to V$ such that $r\in L(f(r))$ for each $r\in R$.

The desired colouring $c$ is now obtained by doing the following. Initialise $c=c'$. Then set $c(f(r))=r$ for every unused rich colour $r\in R\setminus c(V)$. In doing so we removed the original colour from $f(r)$, so this modification may lead to a new (and disjoint) set of unused rich colours. Therefore we repeat the process as long as necessary. That is, while there is a rich colour $r$ unused by $c$, we modify $c$ by setting $c(f(r)) = r$.	When this process terminates, $c$ is a proper $L$-colouring of $K_n$ which uses all the rich colours, as desired.
\end{proof}

For our next result, recall that $K_n^-$ is the graph formed from $K_n$ by removing an arbitrary edge.

\begin{proof}[Proof of Theorem~\ref{thm:n}]
	First observe that $\chi^\star_\ell$ is monotone under edge-deletion. That is, if $G'$ is a subgraph of $G$, then $\chi^\star_\ell(G')\le\chi^\star_\ell(G)$. Thus, for the bound it suffices to prove that $\chi^\star_\ell(K_n)\le n$. In turn, this follows from Lemma~\ref{lem:Kn} with $k=n$. For the equality statement, it suffices to show that $\chi^\star_\ell(K_n^-) \le n -1$, where $n\ge 3$.

	Let $V$ denote the vertex set of $K_n^-$ and let $u_1$ and $u_2$ denote the two nonadjacent vertices in $V$.
	Let $L$ be an $(n-1)$-list-assignment of $K_n^-$.
	If $L(u_1)=L(u_2)$, then we can treat $u_1$ and $u_2$ as the same vertex (since they may take the same colour), and conclude by Lemma~\ref{lem:Kn} for $K_{n-1}$.
	If there is no colour that belongs to all $n$ lists, then we are also done by Lemma~\ref{lem:Kn}.

	Let there be exactly $i$ colours belonging to all $n$ lists. Without loss of generality, we may assume these are $[i]$, i.e.~$[i]\subseteq L(v)$ for all $v \in V$.
	We seek to define an $i$-list-assignment $L_1$ of $K_n^-$ satisfying $L_1(v)\subseteq L(v)$ for all $v\in V$ such that the following properties hold:
	$L_1(u_1)=L_1(u_2)=[i]$;
	every colour in $[i]$ belongs to at most $i+1$ lists $L_1(v)$ where $v\in V$;
	every colour not in $[i]$ belongs to at most $i$ lists $L_1(v)$ where $v\in V$; and
	the $(n-i-1)$-list-assignment $L\setminus L_1$ of $K_n^-$ satisfies the hypothesis of Lemma~\ref{lem:Kn} for $k=n-i-1$.
	If we can do this, then by treating $u_1$ and $u_2$ as the same vertex, we are guaranteed an $L_1$-packing of $K_n^-$ by Lemma~\ref{lem:Kn}.
	We are also guaranteed a disjoint $L\setminus L_1$-packing of $K_n^-$ by Lemma~\ref{lem:Kn}.
	Then these packings combine to form an $L$-packing of $K_n^-$, which completes the proof.

	We prove that such an $L_1$ exists with two claims.
	First, we explain some additional terminology.
	We call a colour $j \not\in [i]$ \emph{rich} if it belongs to more than $n-i-1$ lists $L(v)$ where $v \in V$.
	We write $\alpha(j)$ for the number of lists to which a colour $j$ belongs, and we call $s(j)=\max\{0,\alpha(j)-(n-i-1)\}$ the surplus of $j$. We write $R$ for the set of rich colours, which by definition are the colours with positive surplus.

	\begin{claim}\label{clm:K_n^-Part1}
		We have $\sum_{ j \in R} s(j) \le i(n-i-1)$.
	\end{claim}

	\begin{claimproof}
		Note that for $j\in R$ we have $s(j) \le n-1-(n-i-1) =i$, as by assumption every $j\notin[i]$ belongs to fewer than $n$ lists.
		So if $\lvert R \rvert \le n-i-1$ the claim is immediate.

		In any case, a simple colour-counting argument gives
		\[\sum_{ j \in R } \alpha(j) \le \sum_{ j \not \in [i] } \alpha(j) = n(n-1)-\sum_{ j \in [i] } \alpha(j)=n(n-i-1),\]
		because $[i]\subset L(v)$ for all $v\in V$.
		So in the case $\lvert R \rvert \ge n-i$, we have
		\[\sum_{ j \in R} s(j) =\sum_{ j \in R} \alpha(j) - \lvert R \rvert (n-i-1) \le n(n-i-1)-(n-i)(n-i-1)=i(n-i-1),\]
		completing the proof of the claim.
	\end{claimproof}

	Note that in order for $L\setminus L_1$ to fulfil the condition of  Lemma~\ref{lem:Kn} for $k=n-i-1$, we need that the colour $j$ belongs to at least $s(j)$ lists of $L_1$ for every rich colour $j$. Thus the following claim proves the existence of the desired $L_1$.

	\begin{claim}\label{clm:K_n^-Part3}
		For every $v \in V$, we can select a subset $L_1(v) \subseteq L(v)$ of size $i$ in such a way that $L_1(u_1)=L_1(u_2)=[i]$, every colour in $[i]$ belongs to $i+1$ lists, every colour not in $[i]$ belongs to at most $i$ lists, and every $j\in R$ belongs to at least $s(j)$ lists $L_1(v)$ where $v\in V$.
	\end{claim}

	\begin{claimproof}
		The main part of this proof is to make a partial selection for $L_1$ from the colours not in $[i]$. In particular, we will make a total of $i(n-i-1)$ selections of colours not in $[i]$ from the lists $L(v)$ where $v\in V\setminus \{u_1,u_2\}$, such that no colour is selected more than $i$ times, no vertex has more than $i$ selected colours, and every $j\in R$ is selected at least $s(j)$ times.
		Once we have done this, since $[i]\subseteq L(v)$ for all $v\in V$, the remaining selection for $L_1$ of the colours in $[i]$ is straightforward: we take $L(u_1)=L(u_2)=[i]$, and for the remaining vertices we greedily/arbitrarily add colours of $[i]$ to the selection until $|L_1(v)|=i$ for all $v\in V$ and every colour of $[i]$ has been added precisely $i-1$ times. The resulting selection satisfies the desired conclusion.
		This means that it suffices to make such a selection for $L_1$ of the colours not in $[i]$. For this, we consider two cases.

		First, suppose that $i \ge n-i-1$. We make a somewhat arbitrary partial selection and then modify it.
		We start by taking an arbitary set $A$ of $i$ vertices in $V \setminus \{u_1,u_2\}$, letting $L_1(v)=L(v) \setminus [i]$ for $v\in A$, and letting $L_1(v)=\emptyset$ for all $v\in V\setminus \{u_1,u_2\} \setminus A$.
		Note that $i(n-i-1)$ selections have been made, no colour has been selected more than $i$ times, and no vertex has more than $i$ selected colours thus far.
		While there is a colour $j\in R$ that belongs to fewer than $s(j)$ lists of $L_1$, we can select $j$ for some $v\in V\setminus \{u_1,u_2\} \setminus A$ and deselect for some $v\in A$ some \emph{unnecessary} colour, that is, some selected colour $j'$ such that either $j'\in R$ and $j'$ already appears more than $s(j')$ times or $j'\notin R$.		
Note that at each step of this process, the total number of selections is preserved, the selected colour $j$ is still selected at most $s(j)\leq i$ times in total, and still each vertex has at most $|L(v)\setminus [i]|= n-1-i \leq i$ selected colours.		
		Moreover, Claim~\ref{clm:K_n^-Part1} ensures that this process will terminate, at which point we have the desired partial selection for $L_1$.

		Second, suppose that $i< n-i-1$. Here we make a more careful selection before modifying it. We start by taking some set $A$ of $n-i-1$ vertices in $V \setminus \{u_1,u_2\}$ and consider the $(n-i-1)$-list-assignment $L'$ of the complete graph on vertex set $A$ that is defined by $L'(v)=L(v)\setminus [i]$ for all $v\in A$.  By Lemma~\ref{lem:Kn}, there is a proper $L'$-packing of the complete graph on $A$. Define a partial selection for $L_1$ by arbitrarily taking $i$ of the proper $L'$-colourings from this $L'$-packing and taking their union. Again $i(n-i-1)$ colour selections have been made, no vertex has more than $i$ selected colours and, since we have taken $i$ disjoint proper $L'$-colourings of a complete graph, no colour has been selected more than $i$ times so far.

In what follows, we will modify the partial selection $L_1$ iteratively, by selecting and deselecting some vertices, while at each step maintaining the above three constraints on the number of selections. The goal is of course to make sure that $L_1$ in the end also satisfies the fourth constraint, that every $j\in R$ is selected at least $s(j)$ times. 

At each step, we distinguish between colours in $L_1$ that are merely selected for their vertex, and colours that are also \emph{fixed}. If a colour has been fixed for a vertex, we do not allow it to be deselected from that vertex later in the process. 
Throughout the process, we will only fix rich colours, and each rich colour $j\in R$ will be fixed at most $s(j)$ times, thus ensuring that the total number of times we fix a colour does not exceed $\sum_{j \in R} s(j)$. 

		For the first round of modifications, we first focus our attention on the set $R_1$ of rich colours that belong to at most $n-i-2$ lists $L(v)$ where $v \in V \setminus \{u_1,u_2\}$. A colour $j$ only belongs to $R_1$ when $\alpha(j)=n-i$ (so $s(j)=1$) and $j \in L(u_1) \cap L(u_2)$.
		Since $L(u_1) \ne L(u_2)$, $|R_1|\le n-i-2$.
		But since all of these colours in fact belong to exactly $n-i-2$ lists $L(v)$ for $v \in V \setminus \{u_1,u_2\}$, there exists some injective function $w: R_1\to V \setminus \{u_1,u_2\}$ such that $L(w(j)) \ni j$ for all $j\in R_1$. We now use the function $w$ to iteratively modify $L_1$ so that every colour in $R_1$ is selected (at least) once.		
	More fully, let us consider every $j\in R_1$ and do the following in order, where we write $j\in L_1$ if the colour $j$ has been selected under (the current) $L_1$ for at least one vertex. 
\begin{enumerate}
\item For all $j\in R_1$ such that $w(j) \in A$ and $j \in L_1(w(j))$, we fix colour $j$ for $L_1(w(j))$. 
\item Iteratively, for all $j \in R_1$ such that $w(j) \in A$, but $j \notin L_1$, we select and fix $j$ for $w(j)$, and deselect an arbitrary colour from $L_1(w(j))$.
\item For each remaining $j \in R_1$ such that $w(j)\in A$ we have that $j \in L_1$ but not $j\in L_1(w(j))$. We fix $j$ for one arbitrary selection under $L_1$.
\item Each remaining $j\in R_1$ satisfies $w(j) \notin A$. We (select and) fix colour $j$ for vertex $w(j)$. If as a result of this, there are now more than $i$ vertices for which colour $j$ has been selected, we deselect $j$ for one of them (for which $j$ is not fixed). Otherwise we deselect an arbitrary nonfixed colour.
\end{enumerate}
 After this first round of modifications, every colour in $R_1$ has exactly one fixed selection under $L_1$.
Note that the four modification steps above are rather delicate because we need to preserve the total number of selections, and the properties that no vertex has more than $i$ selected colours and that no colour has been selected more than $i$ times.

		Now we consider the set $R\setminus R_1$ of rich colours that belong to at least $n-i-1$ lists $L(v)$ where $v\in V\setminus \{u_1,u_2\}$. We repeat the following until we can no longer do so.
		Take some $j\in R\setminus R_1$ that currently has been selected by $L_1$ for fewer than $s(j)$ vertices. Let us \emph{fix} all of its current selections. Let $w'(j)\in V\setminus \{u_1,u_2\}$ be some vertex such that $L(w'(j)) \ni j$ and $w'(j)$ does not have $i$ selected colours fixed. The existence of $w'(j)$ is guaranteed by the fact that $j$ belongs to at least $n-i-1$ lists and Claim~\ref{clm:K_n^-Part1}.
		
(Indeed, suppose such $w'(j)$ does not exist. Then there must be at least $n-i-1$ vertices $w\in V\setminus\{u_1,u_2\}$ with $j \in L(w)$ and with at least $i$ colours fixed for $w$. So there are at least $i(n-i-1)$ fixed colours in total. On the other hand, by assumption fewer than $s(j)$ vertices have been fixed for $j$, so in total we have fixed at most $-1+\sum_{h \in R} s(h)$ colours, which is strictly less than $i(n-i-1)$ by Claim~\ref{clm:K_n^-Part1}, a contradiction.)

		As before, we select $j$ for $w'(j)$, \emph{fixing} this choice, and deselect some other (nonfixed) colour, possibly within the list of $w'(j)$ (if we need to ensure that $w'(j)$ has no more than $i$ selected colours). This procedure will terminate due to Claim~\ref{clm:K_n^-Part1}, at which point we will have the desired partial selection for $L_1$.
	\end{claimproof}

	This completes the proof that there exists such an $L_1$, thus guaranteeing the desired $L$-packing.
\end{proof}

\section{Degeneracy}\label{sec:degeneracy}

In this section we give the argument for Theorem~\ref{thm:degeneracycorr}, which also yields Theorem~\ref{thm:degeneracy}.
This greedy approach with Hall's theorem is given by MacKeigan~\cite{Mac21} and discussed by Yuster~\cite{Yus21} as a way of proving that $\chi_c^\star(K_n)\le 2(n-1)$.
Although Yuster proved that the leading constant `$2$' there is not tight, we show with Proposition~\ref{prop:degeneracyDP} that when the method is refined and stated in terms of degeneracy it is indeed tight.
We also give an example showing that the bound in Theorem~\ref{thm:degeneracy} cannot be improved to $\delta^\star(G)+1$.

\begin{proof}[Proof of Theorem~\ref{thm:degeneracycorr}]
	Let $G$ be a graph with degeneracy $\delta^\star(G)=d$, and order the vertices $V(G)$ as $v_1,\dotsc,v_n$ such that no vertex has more than $d$ neighbours preceding it in the order.
	That is, in each of the induced subgraphs $G_i := G[\{v_1,\dots,v_i\}]$, the degree of $v_i$ is at most $d$.
	Let $(L,H)$ be a $k$-fold correspondence-cover of $G$, where $k=2d$.
	We prove by induction and Hall's theorem that $G$ admits a proper $L$-packing. Write $H_i$ and $L_i$ for the correspondence-cover and associated correspondence-assignment induced on $G_i$ by $H$ and $L$, respectively.
	Note that trivially $G_1$ admits a correspondence $L_1$-packing. So for the induction let $i>1$ and assume that $G_{i-1}$ admits a correspondence $L_{i-1}$-packing.
	This packing is equivalent to an ordered list $(c_1, \dotsc, c_k)$ of pairwise-disjoint proper $L_{i-1}$-colourings of $G_{i-1}$.
	We want to extend each of these colourings to the vertex $v_i$ so that they remain disjoint and are proper $L_i$-colourings of $G_i$.
	Since $v_i$ has degree at most $d$ in $G_i$, each colour $x \in L(v_i)$ is a possible valid extension for at least $|L(v_i)|-d=d$ of the colourings $c_1,\dots,c_k$ to the vertex $v_i$.
	For each $x \in L(v_i)$, let $\Xi(x)$ be the set of such colourings (so $|\Xi(x)|\geq d$ for all $x$).
	Conversely, for each $1\leq j \leq k$ there are also at least $d$ colours in $L(v_i)$ that can be used to extend the colouring $c_j$, so $|\Xi^{-1}(c_j)| \geq d$ for all $j$. Consider the family $\mathcal{F} = \{\Xi(x)\,\mid\, x\in L(v_i)\}$. Using the facts that $|\Xi(x)|\ge d$ for all $x$ and $|\Xi^{-1}(c_j)| \geq d$ for all $j$, one can verify that the marriage condition holds. Let $f$ be the injective function guaranteed by Hall's theorem. This gives a one-to-one correspondence between the colourings $c_1,\dots,c_k$ and colours in $L(v_i)$ that may extend them to the vertex $v_i$. We obtain the claimed correspondence $L_i$-packing of $G_i$, as desired. In particular, this holds for $i=n$.
\end{proof}

\begin{proposition}\label{prop:degeneracyDP}
	For every $d\ge 2$, there exists some graph $G$ satisfying $\delta^\star(G)=d$ and $\chi^\star_c(G) = 2d$. In particular, this is true for the complete bipartite graph $K_{d,((2d-1)!)^{d-1}}$.
\end{proposition}

\begin{proof}
	Let $G$ be $K_{d,((2d-1)!)^{d-1}}$, with parts $A$ and $B$ of size $d$ and $((2d-1)!)^{d-1}$, respectively. This graph has degeneracy $d$, so by Theorem~\ref{thm:degeneracycorr} we need to show that it satisfies $\chi^\star_c(G) \ge 2d$. We form a $(2d-1)$-fold correspondence-cover of $G$ with respect to $L$ by including every possible combination of $d$ matchings running between $d$ copies of $[2d-1]$ and a single copy of $[2d-1]$. Then, whatever the partial $L$-packing  of size $2d-1$ induced on $A$, there is guaranteed to be a vertex $b\in B$ and $d$ of the proper $L$-colourings such that $d$ of the colours of $L(b)$ are each matched to $d$ colours of $\cup_{a\in A} L(a)$ that belong to each of the $d$ proper $L$-colourings. And thus to extend the partial $L$-packing induced on $A$, those $d$ colours of $L(b)$ may only be used to extend $d-1$ of the proper $L$-colourings, which is clearly impossible.
\end{proof}

Having established this tight characterisation of $\chi_c^\star$ in terms of degeneracy, it is natural to wonder what the right bound for list packing could be.
We conclude this section with a modest lower bound on the extremal list packing number in terms of degeneracy.
In light of Proposition~\ref{prop:degeneracyDP}, we note on the other hand that (complete) bipartite graphs cannot certify a better lower bound for $\chi_\ell^\star(G)$, see Lemma~\ref{lem:bipgreedylist} below.

\begin{theorem}\label{thm:degeneracyexample}
	For every $d \ge 2$, there exists a graph $G$ with $\delta^\star(G)=d$ and $\chi_{\ell}^{\star}(G) \ge d+2$.
\end{theorem}

\begin{proof}
    We will iteratively construct a graph $G$ with $\delta^\star(G)=d$ and a $(d+1)$-list-assignment $L$ such that $G$ does not admit a proper $L$-packing.
	We will do so by constructing a sequence of subgraphs $G_1,G_2,\ldots, G$ such that $V(G_1) \subset V(G_2) \subset \ldots \subset V(G)$. We will start with a list-assignment and arbitrary list-packing of $G_1$ (see below for more details). Next, for each $G_{j+1}$ apart from the final graph $G$, we consider the list-packing $C:=(c_i)_{1\leq i \leq d+1}$ of $G_j$ and we choose the lists of $V(G_{j+1})\setminus V(G_j)$ such that there is a unique extension of $C$ to $G_{j+1}$.
	For the final graph $G$, we choose the remaining lists such that $C$ cannot be extended.

	We start with choosing $G_{1}=(V_1,E_1)=K_{d+1}$ and the associated lists being equal to $[d+1]$ for all its vertices.
	Note that for every $v \in G_1$, the vector of colours $C(v)=( c_i(v))_{1\le i \le d+1}$ is a permutation of $[d+1]$.
	We now construct $G_2$ by adding a copy $v'$ for every $v \in V_1$ that is connected to all vertices in $V_1 \setminus v.$
	Let $V_2=V(G_2) \setminus V_1.$
	We let $L(v')=[d+1]\setminus \{1\} \cup \{d+2\}$.
	Note that when extending the $L$-packing, the vector $C(v')$ will be equal to $C(v)$ where $1$ is replaced by $d+2,$ as every colour $2 \le j \le d+1$ can only be used for one of the $L$-colourings $c_i$.

	We will repeat this procedure.
	In step $m$, we add copies $v'$ for every $v \in V_m$ and connect it to all vertices in $V_m \setminus v$. we call the set of added vertices $V_{m+1}.$
	For $v' \in V_{m+1}$, we let $L(v')=\s_{ij}(L(v))=L(v) \setminus \{i\} \cup \{j\} $ for some $i,j \in [d+2],$ i.e.~an $(i,j)$-shift is applied to the lists.
	Here we set $V_m=\{v_1^m,\ldots, v_{d+1}^m\}$, where $v_i^{m+1}$ denotes the copy of $v_i^m$.

	By choosing the shifts to be $\s_{1,d+2},\s_{2,1},\s_{d+2,2}$ in the first three steps, we see that in the $L$-packing, the permutations $C(v_i^1)$ and $C(v_i^4)$ are equal up to an interchange of $2$ and $1$, i.e.~the permutation $(12)$ has been applied to them.
	We can repeat this procedure and form the permutation $(123\ldots d)$ as it is equal to $(12),(13),\ldots,(1d)$, a sequence of $d-1$ transpositions.
	So in the $L$-packing, $C(v_i^1)$ and $C(v_i^{3(d-1)+1})$ are two permutations that are equal up to performing the cyclic permutation $(123\ldots d)$.

	Now continue doing this another $d-2$ times.
	Now $C(v_i^{3p(d-1)+1})$ for $0 \le p \le d-1$ are all permutations of $[d+1]$ and every colour in $[d]$ has been used at all spots except one (where $d+1$ is placed every time).
	Now add a final vertex $x$ and connect it to $v_1^{3p(d-1)+1}$ for every $0 \le p \le d-1$, and let $L(x)=[d+1]$ as well.
	A proper $L$-packing of this final graph is impossible.

	In all steps, we connected new vertices to exactly $d$ existing vertices and so the degeneracy of the construction satisfies $\delta^{\star}(G)=d$.
\end{proof}

\section{Maximum degree}\label{sec:maxdegree}

In this section, we give the argument that proves Theorems~\ref{thm:maxdegree} and~\ref{thm:maxdegreecorr}.
The proof closely follows an argument of Aharoni, Berger and Ziv for a result on strong colouring~\cite{ABZ07}; see also the survey of Haxell~\cite{Hax11}.
We only give the argument for correspondence packing as the proof for list packing is essentially the same.
We note that Loh and Sudakov~\cite{LoSu07} outlined this same argument in a slightly more general context, namely, in showing that $\Lambda^\star_c(D) \le 2D+o(D)$ as $D\to\infty$.

\begin{proof}[Proof of Theorem~\ref{thm:maxdegreecorr}]
	Let $k:= 1 + \Delta(G) +  \chi_c(G)$.  Let $H$ be a $k$-fold correspondence-cover of $G$, via some correspondence-assignment $L:V(G) \rightarrow 2^{V(H)}$. Recall that this means that the vertices of $H$ are partitioned into parts $(L(v))_{v\in V(G)}$ such that $|L(v)|=k$ for all $v\in V(G)$, and if $vw \in E(G)$ there is a (possibly empty) matching between $L(v)$ and $L(w)$, and for every $vw \notin E(G)$ there is no edge between $L(v)$ and $L(w)$.

	We need to show that $G$ admits a correspondence $L$-packing. Equivalently, we need to find a proper $k$-colouring of $H$ such that for every $v\in V(G)$, no two vertices of $H$ in the same part $L(v)$ receive the same colour (recall that each colour class of such a $k$-colouring corresponds to a single $L$-colouring in the desired correspondence $L$-packing).
	To this end, let $\alpha$ be a \emph{maximum partial colouring} of $H$, i.e.~a proper $k$-colouring of the vertices of $H$ using as many vertices of $H$ as possible.
	If $\alpha$ assigns a colour to every vertex of $H$ then we are done.
	So suppose for contradiction that $\alpha$ does not assign a colour to some vertex $x$ of $H$ in some part $L(v_1)$. Then some colour in $\{1,\ldots, k\}$ is missing from $L(v_1)$; let us call it \emph{red}.

	Let $v_1,\ldots, v_n$ denote the vertices of $G$. For each $2 \leq i \leq n$, let $r_i$ denote the red vertex in $L(v_i)$ if it exists, and define  $L'(v_i):= L(v_i) \setminus \{y : \alpha(y)=\alpha(z)   \text{ for some } z \in N_H(r_i)\}$. Thus, $L'(v_i)$ consists of those vertices of $L(v_i)$ whose colour could be reassigned to $r_i$ without creating a conflict. Also let $L'(v_1)=L(v_1)$.
	Then $|L'(v_i)|\geq \chi_c(G)+1$ for every $i$.

	Next let $L''(v_i)= L'(v_i) \setminus N_H(x)$.  Since $x$ has at most one neighbour in each $L(v_i)$, it follows that $|L''(v_i)|\geq \chi_c(G)$ for each $1\leq i \leq n$, and thus by definition of correspondence colouring, $(L''(v_i))_{1\leq i \leq n}$ has an independent transversal $T''$. By replacing the representative of $T''$ in $L''(v_1)$ with $x$, this in turn yields an independent transversal $T'$ of $(L'(v_i))_{1\leq i \leq n}$ which contains $x$.
	Next, we modify $\alpha$ by giving colour red to every vertex of $T'$, and for each $i$ for which $r_i$ exists and $r_i \notin T'$ we give $r_i$ the colour of the element $t_i$ of $T'$ in $L(v_i)$. By definition of $L'(v_i)$, this does not create any conflict, so we obtain a valid $k$-colouring $\beta$. Moreover, $\beta$ colours $x$ as well as all the vertices that were coloured by $\alpha$, contradicting the maximality of $\alpha$.
\end{proof}

\section{On the permanent of random binary matrices}\label{sec:randommatrix}

In this section, we describe and derive auxiliary results on transversals and permenants of random binary matrices. These are our key tools for the proofs of Theorems~\ref{thm:fractional} and~\ref{thm:hallratio}.
Related ideas are also useful for Theorem~\ref{thm:bipartite}, which we elaborate upon in Section~\ref{sec:bip}.

Our starting point is a theorem of Everett and Stein~\cite{ES73} from 1973. For the benefit of the reader, we will use this below in a warm-up to prove the special case for Theorem~\ref{thm:fractional} of list packing bipartite graphs with $n$ vertices (Theorem~\ref{th:bipartiteLogN}).

\begin{theorem}[\cite{ES73}]\label{thm:EverettSteijn}
	The fraction of binary $k \times k$ matrices $A$ with $\Per(A)=0$  is $(1+o(1))k 2^{1-k}$, as $k \rightarrow \infty$.
\end{theorem}

Theorem~\ref{thm:EverettSteijn} can be interpreted probabilistically: a uniformly random binary matrix has permanent $0$ with probability $(1+o(1)) k 2^{1-k}$ as $k\rightarrow \infty$. For the proofs of Theorems~\ref{thm:fractional} and~\ref{thm:hallratio}, we will need a more general result that allows for an asymmetric treatment of the ones and zeroes. The following is a generalisation of Theorem~\ref{thm:EverettSteijn}.

\begin{theorem}\label{thr:Permenant_BernoulliDistributedMatrix}
	Let $0\leq p<1$ be a real number.
	Let $A$ be a random binary $k \times k$-matrix for which each entry independently equals $0$ with probability $p$, and $1$ with probability $1-p$.
	Then
	\[\Per(A)=0 \text{ with probability } 2k p^k (1+o(1)) \text{ as } k \to \infty.\]
	Moreover, the same conclusion holds if $p$ is allowed to depend on $k$, provided $k^5 p^k \rightarrow 0$ as $k\rightarrow \infty$. In particular this holds if $p \leq 1- \frac{5 \log k+ \omega(1)}{k}$ as $k\rightarrow \infty$.
\end{theorem}


In the proof of Theorem~\ref{thr:Permenant_BernoulliDistributedMatrix}, we will use the following technical lemma.

\begin{lemma}\label{lem:comp_estimate}
	Let $0 \leq p <1$ be such that $k^5p^k \rightarrow 0$ as $k \rightarrow \infty$. Then as $k \rightarrow \infty$,
	\[\max_{\substack{s, t \in \mathbb{N} \text{ such that}\\ 2 \leq s \leq t \leq k \text{ and }\\s+t \geq k}}\left\{ \binom{k}{s} \binom{k}{t} p^{st} \right\} = o\left(\frac{1}{k} p^k\right).\]
\end{lemma}

\begin{proof}

	Let $2 \leq s \leq t \leq k$ be such that $s+t\geq k$.

	If $s \ge \frac k2$, we have
	\[\binom{k}{s} \binom{k}{t} p^{st} <4^k p^{k^2 /4} =  o\left(\frac{1}{k} p^k\right).\]

	If $2\le  s < \frac k2$, then

	\begin{align*}
		\frac{k}{p^k} \binom{k}{s} \binom{k}{t} p^{st} & \le \frac{k}{p^k} \binom{k}{s} \binom{k}{k-s} p^{s(k-s)}                    \\
		                                               & \le k^{2s+1} p^{  (s-1)k-s^2 }                                              \\
		                                               & =  \left( k^5 p^k \right)^{\frac{2s+1}{5}} \cdot p^{\frac{3k(s-2)}{5} -s^2} \\
& \leq  k^5 p^k  \cdot p^{\frac{3(2s+1)(s-2)}{5} -s^2} \\		                                              
		                                               & \leq  k^5 p^{k-6}= o(1)
	\end{align*}
	as $k \rightarrow \infty$.
	%
	%
	%
	%
\end{proof}

\begin{proof}[Proof of Theorem~\ref{thr:Permenant_BernoulliDistributedMatrix}]
	By Lemma~\ref{le:distinctrepr}, if $\Per(A)=0$ then there are sets $S$ and $T$ with $s$ and $t$ elements respectively for which $s+t>k$ and the submatrix $B=A_{S\times T}$ only contains zeros. In particular, we have such a submatrix $B$ with $s+t=k+1$.
	Let us call such $B$ a \emph{$t$-bad submatrix}.

	Given $s$ and $t$, there are $\binom{k}s \binom kt$ choices for the sets $S,T$.
	The probability that $A_{i,j}=0$ for all $ i\in S , j \in T$ equals $p^{st}.$ So the probability that there is a $t$-bad submatrix is upper bounded by $\binom{k}{k+1-t} \binom{k}{t} p^{(k+1-t)t}$.

	We now take the union bound over all $1\leq t \leq k$.
	The cases $t=1$ and $t=k$ each contribute $kp^k$, while each $2\leq t \leq k-1$ contributes $o\left(\frac{1}{k} p^k\right)$ by Lemma~\ref{lem:comp_estimate}. 	Hence the probability that there exists a $t$-bad submatrix for some $1\leq t \leq k$ is at most $2kp^k + (k-2) \cdot o\left(\frac{1}{k} p^k\right)=(1+o(1))2kp^k$, as desired.

	On the other hand, we know that any matrix with a row or column with only zeros has permanent equal to $0$.
	Using a truncated version of the principle of inclusion--exclusion, i.e.~subtracting the probabilities that are counted twice, we note that the probability that this happens is at least
	$2kp^k -( 2\binom{k}{2}p^{2k} + k^2 p^{2k-1} )=(1-o(1))2kp^k,$ as desired.
\end{proof}


\paragraph{Remark.}

In Theorem~\ref{thr:Permenant_BernoulliDistributedMatrix}, the opposite random variables $1-A_{i,j}$ are independent and Bernoulli($p$) distributed. We note that the \emph{upper} bound $2k p^k (1+o(1))$ in Theorem~\ref{thr:Permenant_BernoulliDistributedMatrix} remains true with the same proof if instead of independence, we only require that the random variables $1-A_{i,j}$ are negatively correlated. 
 Indeed, in this case we still have for each subset of rows $S$ and columns $T$ of $A$ that the probability that $A_{i,j}=0$ for all $ i\in S , j \in T$ is at most  (rather than equal to) $p^{st},$ which is sufficient.

\section{Fractional chromatic number and Hall ratio}\label{sec:approx}

In this section we prove Theorems~\ref{thm:fractional} and~\ref{thm:hallratio}, using the results on random binary matrices from Section~\ref{sec:randommatrix}.
As a warm-up using the classic result of Everett and Stein, we first derive the special case of Theorem~\ref{thm:fractional} for bipartite graphs.

\begin{theorem}\label{th:bipartiteLogN}
	Let $G$ be a bipartite graph on $n$ vertices. Then as $n \rightarrow \infty$,
	\[\chi_\ell^\star(G) \leq  (1+o(1))  \log_2 n.\]
\end{theorem}
\begin{proof}
	Let $\eps>0 $ and set $k$ to be an integer that is at least $(1+\eps) \log_2 n$. Let $L$ be an arbitrary $k$-list-assignment for $G$.

	We consider the union of all lists, $\mathcal{L} = \bigcup_{v \in V(G)} L(v)$, and colour each list element $\ell \in \mathcal{L}$ independently with a uniformly random vector $x_{\ell} \in\{0,1\}^{k}$.
	We use these vectors to associate a random binary matrix with each vertex of $G$, as follows. For each $v\in V(G)$, let $M(v)$ be a $k \times k$ matrix formed by concatenating the vectors $(x_{\ell})_{\ell \in L(v)}$ in some arbitrary way.

	Let $V_0,V_1\subseteq V(G)$ denote the two parts of the bipartition of $G$.
	Now, for each vertex $v\in V(G)$ we define a \emph{bad event}  $\mathcal{B}(v)$ as follows: if $v \in V_0$, then $\mathcal{B}(v)$ is the event that $M(v)$ has no $0$-transversal, and similarly if $v \in V_1$, then $\mathcal{B}(v)$ is the event that $M(v)$ has no $1$-transversal.

	If there are no bad events, then for each $v\in V_0$ we can choose a $0$-transversal of $M(v)$, and for each $v\in V_1$ we can choose a $1$-transversal of $M(v)$.  These transversals then determine the desired $k$ disjoint $L$-colourings $c_1,c_2, \ldots, c_k$ as follows: for $1\leq i \leq k$ and $v\in V(G)$, we set $c_i(v)$ to be the element of $L(v)$ that corresponds to the $i$-th element of the transversal of $M(v)$. Since we choose the colours of $v$ according to a transversal of $M(v)$, we have $c_i(v) \neq c_j(v)$ for any $i\neq j$, so the colourings $c_1,\ldots, c_k$ are indeed disjoint. Furthermore, for each $i$ the colouring $c_i$ is proper because the colour sets on $V_0$ and $V_1$ are disjoint; indeed, any vertex $v\in V_0$ (resp.\ $v\in V_1$) is coloured by an element $l \in \mathcal{L}$ such that the $i$-th element of $x_l$ is $0$  (resp.\ $1$).

	Thus, it suffices to find a colouring without bad events. For $v\in V_1$, the probability of $\mathcal{B}(v)$ equals the probability that a uniformly random binary matrix has no $1$-transversal.
	By Theorem~\ref{thm:EverettSteijn}, the number of binary $k\times k$ matrices that have no $1$-transversal is at most $(1+o(1)) k 2^{k^2-k+1}$ as $n\rightarrow \infty$. Hence, $\Pr(\mathcal{B}(v)) \leq (1+o(1))\cdot \frac{k \cdot 2^{k^2-k+1}}{ 2^{k^2}}  = (2+o(1)) k2^{-k}$.
	By symmetry, the same bound holds for every $v\in V_0$. Therefore, by a union bound, the probability that at least one bad event holds is at most $n \cdot (2+o(1)) k2^{-k}$, and for $n$ large enough this is strictly smaller than $1$. Thus, there exists a colouring of $\mathcal{L}$ without bad events and hence the desired $L$-packing.
\end{proof}

We now proceed to prove Theorems~\ref{thm:fractional} and~\ref{thm:hallratio}, thus extending the result on bipartite graphs (Theorem~\ref{th:bipartiteLogN}) to graphs with bounded fractional chromatic number or bounded Hall ratio. This includes the special case of multipartite graphs and generalises bounds on $\chi_\ell$ from~\cite{CJKP20} and~\cite{Alo92}.

\begin{theorem}\label{thm:fractional_restated}
	The graphs $G$ with $n$ vertices and at least one edge satisfy as $n\rightarrow \infty$
	\[\chi_\ell^\star(G) \le (5+o(1)) \frac{\log n}{\log(\chi_f(G)/(\chi_f(G)-1))}\le (5+o(1)) \cdot \chi_f(G) \log n.\]
	Moreover, for every fixed rational number $m$,  the graphs $G$ with $\chi_f(G)=m$ and $n$ vertices and at least one edge satisfy as $n\rightarrow \infty$
	\[\chi_\ell^\star(G) \le (1+o(1))\cdot  \frac{\log n}{\log(m/(m-1))} \le (1+o(1))\cdot m \log n .\]

\end{theorem}

\begin{proof}
	We give the proof for the first bound (where $\chi_f(G)$ is allowed to grow with $n$), and at the end we detail why almost the same proof also yields the second bound (for $\chi_f(G)$ bounded by a constant independent of $n$).

	We write $m:=\chi_f(G)$.
	Since the clique number is a lower bound for fractional chromatic number and $G$ has at least one edge, we have $m\geq 2$.

	Let $\eps>0$ and consider a $k$-list-assignment $L$ with $k= \lfloor \frac{5+\eps}{\log(m/(m-1))} \cdot \log n \rfloor$. We may assume that $k<n$, since otherwise there is an $L$-packing due to Theorem~\ref{thm:n}.
	We also note that $k \rightarrow \infty$ if and only if $n \rightarrow \infty$; this is where we use that $m \geq 2$.

	Let $\mathcal{L}=\bigcup_{u \in V(G)}L(u)$ be the set of all colours.
	For a positive integer $a$, recall that $[a]$ is the set $\{1,\ldots, a\}$.
	Let $c$ be a proper $(a,b)$-colouring of $G$ with $\frac ab = m$. By definition, this means that $c: V(G) \rightarrow [a]^b$ is a function such that $c(u) \cap c(v)=\emptyset$ for every two adjacent vertices $u$ and $v$.
	For every $\ell \in \mathcal{L}$, let $x_{\ell}$ be a uniformly random vector in $[a]^k.$
	For each $v\in V(G)$, let $M(v)$ be a $k \times k$ matrix formed by concatenating the vectors $(x_{\ell})_{\ell \in L(v)}$ in some arbitrary way.

	For every vertex $v\in V(G)$, we define the bad event $\mathcal{B}(v)$ as the event that the $[a]$-valued $k \times k$ matrix $M(v)$ has no transversal with only elements from $c(v)$.
	Consider the binary matrix $M^*(v)$ obtained from $M(v)$ by replacing every element that belongs to $c(v)$ with a $1$, and every other element with a $0$.

	Note that $M^*(v)$ is a random binary $k \times k$ matrix for which each element independently equals $0$ with probability $p=\frac{a-b}{a}=\frac{m-1}{m}$.
	Therefore, we wish to apply Theorem~\ref{thr:Permenant_BernoulliDistributedMatrix} to $M^*(v)$
	with $p=\frac{m-1}{m}$.
	For that, we first need to verify the condition 
	that $k^5p^k=o(1)$ as $k \rightarrow \infty$. This is satisfied because
	\begin{align*}
		k^5 \cdot p^k & <  n^5 \cdot \left( \frac{m-1}{m} \right)^{ \frac{(5+\eps)\log n}{\log(m/(m-1))} - 1}                                    \\
		              & \leq n^5  \cdot \exp \left( \frac{(5+\eps) \log n}{\log(m/(m-1))}   \cdot \log{\left(\frac{m-1}{m}\right)}  \right) \cdot 2 \\
		              & = n^5  \cdot n^{-(5+\eps)} \cdot 2
	\end{align*}
	goes to $0$ as $n \rightarrow \infty$.

	So we may indeed apply Theorem~\ref{thr:Permenant_BernoulliDistributedMatrix} to $M^*(v)$, which yields that $M^*$ has no $1$-transversal with probability $2kp^k(1+o(1))$. Therefore the probability that $\mathcal{B}(v)$ occurs is $2kp^k(1+o(1))$, as $k \rightarrow \infty$.

	By a union bound, the probability that at least one bad event occurs is at most
	\begin{align*}
		n\cdot 2k\left(\frac{m-1}{m}\right)^{k}\left(1+o(1) \right) & < 2n^2 \cdot \left( \frac{m-1}{m} \right)^{ \frac{(5+\eps)\log(n)}{\log(m/(m-1))}-1} \cdot(1+o(1))                                  \\
		                                                            & \leq 4 n^2  \cdot \exp \left( \frac{(5+\eps) \log(n)}{\log(m/(m-1))}   \cdot \log{\left(\frac{m-1}{m}\right)}  \right) \cdot(1+o(1)) \\
		                                                            & \leq \frac{4}{e \cdot n^3}  \cdot(1+o(1)).
	\end{align*}

	For $n$ large enough (and hence for $k$ large enough), this is smaller than $1$.
	So with positive probability, every $M(v)$ has a transversal with only elements from $c(v)$.

	These transversals then determine the desired $k$ disjoint $L$-colourings $c_1,c_2, \ldots, c_k$ as follows: for $1\leq i \leq k$ and $v\in V(G)$, we set $c_i(v)$ to be the element of $L(v)$ that corresponds to the $i$-th element of the transversal of $M(v)$. Since we choose the colours of $v$ according to a transversal of $M(v)$, we have $c_i(v) \neq c_j(v)$ for any $i\neq j$, so the colourings $c_1,\ldots, c_k$ are indeed disjoint. Furthermore, for each $i$ the colouring $c_i$ is proper because for any two neighbours $u,v \in V(G)$, they are coloured by elements $\ell_u, \ell_v \in \mathcal{L}$ such that the $i$-th element of $x_{\ell_u}, x_{\ell_v}$ belongs to $c(u)$ and $c(v)$ respectively, which are disjoint since $c$ is a proper $(a,b)$-colouring.

	This concludes our proof that  $\chi_\ell^\star(G) \leq (5+o(1)) \cdot  \frac{\log n}{\log(\chi_f(G)/(\chi_f(G)-1))}$. We remark that the condition of Theorem~\ref{thr:Permenant_BernoulliDistributedMatrix} is the main obstruction for finding a better constant; if it were not for that part of the argument, we could have improved the factor $5+o(1)$ to $2+o(1)$.
	To confirm that $5+o(1)$ \emph{can} be improved to $1+o(1)$ in the special case that $\chi_f(G)=m$ is a constant independent of $n$, we need to check that our two asymptotic estimates survive if we instead choose list size $k= \lfloor \frac{(1+\eps) \log n}{\log(m/(m-1))} \rfloor$: the condition of Theorem~\ref{thr:Permenant_BernoulliDistributedMatrix} indeed still holds because \[k^5\cdot p^k \leq  \left(\frac{(1+\eps) \log n}{\log(m/(m-1))}\right)^5  \cdot n^{-(1+\eps)}\cdot 2 \rightarrow 0\] as $n\rightarrow \infty$, and so the probability that a bad event occurs is bounded from above by \[n\cdot 2k \cdot\left(\frac{m-1}{m} \right)^k (1+o(1)) \leq 4n \cdot  \frac{(1+\eps) \log n}{\log(m/(m-1))} \cdot n^{-(1+\eps)} (1+o(1))  \rightarrow 0\] as $n \rightarrow \infty$.
\end{proof}

As an almost immediate corollary, we can derive the following bound, which is reminiscent of a similar result for $\chi_\ell$ instead of $\chi_\ell^\star$ in~\cite{NoPo23}.

\begin{corollary}
	Let $G$ be a graph with order $n$ and Hall ratio $\rho.$ Then \[\chi_\ell^\star(G) \le (5+o(1)) \cdot \rho \log^2 n.\]
\end{corollary}

\begin{proof}
	Iteratively, one can select an independent set of size at least $1/{\rho}$ times the order of the remaining graph.
	Since $n\left(1-\frac1{\rho}\right)^{ \rho \log n}<1$, we have $\chi(G) \le \rho \log n.$
	Hence by Theorem~\ref{thm:fractional_restated}, we can conclude as $\chi_\ell^\star(G) \le (5+o(1)) \cdot \chi(G) \log n.$
\end{proof}

As a corollary of Theorem $1.3$ in~\cite{CJKP20} and Theorem~\ref{thm:fractional_restated}, we immediately get the following Ramsey-type result.
\begin{corollary}
	There exists a constant $C>0$ such that $\chi_\ell^\star(G) \le C \cdot \min\{ \sqrt{ n \log n}, \frac{n \log n}d \} $ for every triangle-free graph $G$ with $n$ vertices and minimum degree $d$.

\end{corollary}

\section{Bipartite graphs}\label{sec:bip}

This section is devoted mainly to the proof of Theorem~\ref{thm:bipartitecorr}, which in turn implies Theorem~\ref{thm:bipartite}.
Before that though, we start with a simple bound on the list packing number of a bipartite graph that resembles the greedy bound (for not-necessarily-bipartite graphs) on list chromatic number but has a somewhat more intricate proof.

\begin{lemma}\label{lem:bipgreedylist}
	Let $G=(A\cup B,E)$ be a bipartite graph with parts $A$ and $B$ having maximum degrees $\Delta_A$ and $\Delta_B$, respectively, with $\Delta_A\le\Delta_B$.
	Then $\chi^\star_\ell(G) \le \Delta_A+1$.
\end{lemma}

\begin{proof}
	Let $k=\Delta_A+1$, and let $L$ be any $k$-list-assignment of $G$.
	To construct a proper $L$-packing of $G$, we first choose the colours of vertices in $B$ for every colouring in the packing.
	To do this, consider the lists of vertices in $B$ in their numerical orderings and let $c_i(b)$ be the $i$-th colour in the list $L(b)$ for each $1\le i\le k$ and each $b\in B$.
	This is the unique proper $L$-packing of $G[B]$ such that the colourings $c_i$ for $1 \le i \le k$ satisfy for all vertices
	$b$ in $B$,  $c_1(b) < c_2(b) < \dotsb < c_k(b)$.

	Next, we prove that we can extend this partial $L$-packing to $A$ by applying Hall's marriage theorem.
	Let $a$ be an arbitrary vertex in $A$.
	For every colour $j\in L(a)$, let $I_j$ be the set of indices $1 \le i \le k$ such that setting $c_i(a)=j$ retains the property that $c_i$ is a proper $L$-colouring (i.e.~no neighbour $b$ of $a$ has $c_i(b)=j$).
	Consider the family $\mathcal{F} = \{ I_j : j\in L(a) \}$. A system $f(\mathcal{F})$ of distinct representatives for $\mathcal{F}$ is precisely an extension of the partial $L$-packing to $a$, as we can set $c_{f(I_j)}(a) = j$ for all $j\in L(a)$.
	For the marriage condition, it suffices to prove that for every $a \in A$ and any subset $J \subseteq L(a)$, we have $\left| \cup_{j\in J} I_j \right| \ge |J|$.
	Suppose for a contradiction that this is not the case. Then there is some $J$ with $\left| \cup_{j\in J} I_j \right| \le |J|-1$, and hence there is a set $Z$ of $k-(|J|-1)$ indices $i$ for which $c_i(a)$ cannot be set equal to any colour in $J$.
	Let $z^\star$ be the largest index in $Z$ and $j^\star$ be the largest colour in $J$.
	There are at least $k-(|J|-1)$ neighbours $b$ of $a$ which have $c_z(b)=j^\star$ for some $z \in Z$.
	Furthermore, by the choice of the colourings $c_i$ these are different from the $|J|-1$ neighbours $b$ that satisfy $c_{z^\star}(b)=j$ for $j \in J \setminus \{ j^\star\}$.

	To see that they are indeed different, suppose for a contradiction that there exists $b \in B$ such that both $c_z(b)=j^{\star}$ and $c_{z^\star}(b)=j$, for some $z \in Z$  and $j\in J \setminus \{ j^\star\}$. As $j<j^\star$ we must have $z \neq z^\star$, so by the choice of the colourings $c_i$ it would follow that $c_z(b)< c_{z^\star}(b)=j<j^\star = c_z(b)$, a contradiction.

	We conclude that $a$ has at least $k> \Delta_A$ neighbours in $B$, which is a contradiction.
	As we can perform this extension for all $a\in A$ independently, this completes the proof.
\end{proof}

\begin{corollary}\label{cor:bipartite}
	When $a$ is sufficiently large in terms of $b$, we have $\chi^\star_\ell (K_{a,b})= b+1$, while $\chi^\star_c (K_{a,b})=2b$.
\end{corollary}
\begin{proof}
	This follows from Theorem~\ref{thm:degeneracycorr}, Proposition~\ref{prop:degeneracyDP} and Lemma~\ref{lem:bipgreedylist}, together with the classic fact that $\chi_\ell(K_{b^b,b})\geq b+1$ for every $b$. To see the latter: assign $b$ disjoint lists $L_1, \ldots, L_b$ of size $b$ to the vertices in the smaller part of the bipartition of $K_{b^b,b}$, and then assign each of the $b^b$ distinct $b$-tuples in $L_1 \times \ldots \times L_b$ as a list to some vertex in the larger part of the bipartition. This list-assignment does not admit a proper list-colouring.
\end{proof}

In very broad terms, the proof of Theorem~\ref{thm:bipartitecorr} is similar to the proof of Lemma~\ref{lem:bipgreedylist}, except that instead of Hall's theorem, we use the Lov\'asz local lemma to complete the extension of a (random) partial (correspondence) packing of $G[B]$.
We will establish the result in the following slightly more refined form.

\begin{theorem}\label{lem:bipartite}
	For every $\eps>0$ fixed, there is $\Delta_0>0$ such that the following holds.
	Let $G=(A\cup B,E)$ be a bipartite graph with parts $A$ and $B$ having maximum degrees $\Delta_A$ and $\Delta_B$, respectively, with $\Delta_0\le\Delta_A\le\Delta_B$. Suppose that $k=\left\lceil (1+\eps) \frac{\Delta_A}{\log \Delta_A} \right\rceil$  satisfies that
	\begin{align*}
		3e\Delta_A(\Delta_B-1) \cdot k^2\exp \left( -\Delta_A^{\eps /3}  \right) <1.
	\end{align*}
	Then $\chi^\star_c(G) \le k$.
\end{theorem}

\noindent
Note that Theorem~\ref{thm:bipartitecorr} follows easily from this result by checking that the condition holds for $\Delta_A=\Delta_B=\Delta$ taken large enough.
The argument to prove Theorem~\ref{lem:bipartite} is inspired by a similar method used in~\cite{ACK21,CaKa22} and a technical extension of the results in Section~\ref{sec:randommatrix}.
Just like in our derivations of Theorems~\ref{thm:fractional} and~\ref{thm:hallratio}, we prove  Theorem~\ref{lem:bipartite} with the aid of a suitably strong bound on the probability that some random matrix has no $0$-transversal. 

\begin{lemma}\label{cor:low_probability_zerotransversal}
	Fix $\eps>0$ and let $k = \lceil (1+\eps) \frac{n}{\log n} \rceil$ for $n$ sufficiently large.
	Let $M$ be a $k \times k$-matrix which is the sum of $n$ independent uniformly random $k \times k$ permutation matrices. Then the probability that there is no transversal in $M$ containing only zeros is smaller than $3k^2\exp \left( -n^{\eps /3}  \right)$.
\end{lemma}

Let us first show how this probability bound is sufficient to derive the main result through a straightforward application of the Lov\'asz local lemma.

\begin{proof}[Proof of Theorem~\ref{lem:bipartite}]
	We may assume that every vertex in $A$ has degree $\Delta_A$ as we can embed $G$ in such a graph.
	Let $H$ be a $k$-fold correspondence-cover of $G$, via some correspondence-assignment $L:V(G) \rightarrow 2^{V(H)}$.
	We may assume that for every $uv \in E(G)$, the matching in $H$ between $L(u)$ and $L(v)$ is a perfect matching.
	To construct a correspondence $L$-packing of $H$, we first define a random partial $L$-packing restricted to the vertices in $B$. To do this, take a uniform total ordering of $L(b)$ independently at random for each $b\in B$, and let $c_i(b)$ be the $i$-th element in $L(b)$ for each $1\le i\le k$. Note that $c_i(B)$ for each $i$ is a uniformly random maximum independent set in $H[L(B)]$.
	We now prove  
	that the probability that the random $L$-packing cannot be extended in a fixed vertex $a$ is very small.

	We define a \emph{bad event} $\mathcal{B}(a)$ for $a\in A$ as occurring if it is impossible to order the elements of $L(a)$ in such a way that for all $i$, the $i$th element of $L(a)$ has no neighbour in $c_i(B)$.
	Let $T_{i,c}(a)$ be the event that $c\in L(a)$ is adjacent in $H$ to a vertex in $c_i(B)$.
	Let $M(a)$ be a $k\times k$ matrix such that the entry $(i,c)$ equals $1$ if $T_{i,c}(a)$ occurs and equals $0$ otherwise.
	Then a good ordering of the elements of $L(a)$ exists precisely when there is a transversal only containing zeros in $M(a)$.
	By applying Corollary~\ref{cor:low_probability_zerotransversal} with $n=\Delta_A$, we conclude that $q := \Pr(\mathcal{B}(a)) < 3k^2 \exp \left( -\Delta_A^{\eps /3}  \right)$ provided $\Delta_0$ is chosen large enough.

	Note also that each event $\mathcal{B}(a)$ is mutually independent of all other events $\mathcal{B}(a')$ apart from those corresponding to vertices $a'\in A$ that have a common neighbour with $a$ in $G$. There are at most $d:= \Delta_A(\Delta_B-1)$ such vertices $a'$ other than $a$.
	With the above choices of $q$ and $d$, we have by assumption that
	\(
		eqd <1,
	\)
	so that the Lov\'asz local lemma guarantees that with positive probability none of the events $\mathcal{B}(a)$ occur. And thus the partial correspondence $L$-packing on $L(B)$ can be extended to all of $H$, as desired.
\end{proof}

It only remains to prove Lemma~\ref{cor:low_probability_zerotransversal}.
Hypothetically this lemma would follow easily from a certain negative correlation property, and so we first discuss some somewhat surprising situations where this property fails.

Let $R$ be the $k \times k$-matrix which records for each of its elements whether it appears in at least one of $n$ independent uniformly random $k \times k$ permutation matrices $P^1,\dotsc,P^{n}$, i.e.~$R_{i,j}=1$ if and only if there is some $1\le \rr \le n$ for which $P^{\rr}_{i,j}=1$ and $R_{i,j}=0$ otherwise. Equivalently, one can consider the sum $M$ of the $n$ permutations matrices, and then replace every nonzero element with a $1$ to obtain $R$.
By Lemma~\ref{le:distinctrepr}, we would like to estimate the probability that a fixed $s \times t$ submatrix of $R$ contains no zero, and in fact we would want to aim for $p^{st}$ as an upper bound, where $p$ is the probability that one entry is nonzero. This would be immediate if all entries of $R$ were independent or negatively correlated, but it turns out this is not the case.

Note that within one row of $R$, the values of its entries are negatively correlated, so the desired upper bound $p^{st}$ indeed holds in the special case $s=1, t=k$.
Intuitively, one might expect this negative correlation to survive in a $s\times t$ submatrix for other values of $s$ and $t$. Unfortunately, it is already false for $s=t=k=n=2$.

More surprising is the fact that the probability that the entries in $R_{[2]\times[2]}$ are all nonzero can be larger than $\Pr(R_{1,1}\not=0)^4$ for larger values of $k$ and $n$ as well, e.g.\ when $(k,n)=(4,32)$.
This has been computed for certain values with Maple\footnote{The relevant Maple code can be accessed in the document CounterexamplesNegativeCorrelationSumPermutationMatrices.mw at \url{https://github.com/StijnCambie/ListPack}.}.

The obstructions discussed above could help the reader understand how our  computations ended up being a bit more complicated than we might have hoped.

We now state the main technical engine in the proof of Theorem~\ref{thm:bipartitecorr}, in which we compute an upper bound for the probability that a given submatrix of a sum of random permutation matrices is nowhere zero.
In the proof it will be useful to denote by $S(M)$ the sum of all entries in a matrix $M$.

\begin{lemma}\label{lem:low_probability_nonzero_l*msubmatrix_sumpartialpermutations_l}
	Let $M$ be a $k \times k$-matrix which is the sum of $n$ independent uniformly random $k \times k$ permutation matrices $P^{\rr},$ $1\le \rr \le n$.
	For a fixed $s \times t$-submatrix of $M$, with $s \ge t$ and $t \le \frac{k}{4}$, the probability that all its elements are nonzero is at most
	\[2t \cdot \exp \left( -\frac{\delta^2}{3} \frac{s n}{k} \right)+\left[ 4 \exp \left( -\frac{\delta^2}{12}\left(1-\frac\delta3\right)  \frac{s n}{k} \right) + \exp\left(- s \exp\left( -\left(1+\frac{\eps}3\right)\left(1+\delta\right)\frac{n}{k} \right)\right) \right]^t \] for every choice of $\delta=\delta(n,k) \in [0,1]$, $\eps >0$ and $k$ sufficiently large as a function of $\eps$.
\end{lemma}

\begin{proof}
	By symmetry and to ease notation, it is sufficient to prove the case where the $s \times t$-submatrix is $[s] \times [t]$.
	First we fix $j \in [t]$ and  estimate $S(M_{[s] \times j})$, the sum of elements in the $[s] \times j$-submatrix of $M$.

	Let $T_j$, $1 \le j \le t$, be the event that $(1- \delta) \frac{sn}{k} \le S(M_{[s] \times j}) \le (1+ \delta) \frac{sn}{k}.$
	Let $T=\{T_1,T_2, \ldots, T_t\}$ be the event that all of them hold, and similarly we write $T \setminus T_j=\left\{T_1,\ldots, T_{j-1},T_{j+1},\ldots, T_t  \right\}$.
	With $\overline T_j$, we denote the event that $T_j$ does not hold.
	Note that the random variables $P^{\rr}_{i,j}$ and $P^{\rr'}_{i',j}$ are independent if $\rr' \not=\rr$.
	Moreover, for fixed $\rr$ and $j$,  we have by Lemma~\ref{lem:negcorr} that $P^{\rr}_{1,j}, P^{\rr}_{2,j},\ldots, P^{\rr}_{k,j}$ are negatively correlated,
	and that the \emph{opposite random variables} $1-P^{\rr}_{1,j}, 1-P^{\rr}_{2,j},\ldots, 1-P^{\rr}_{k,j}$ are negatively correlated as well.
	Thus \[ S(M_{[s] \times j})=\sum_{\rr=1}^n \sum_{i=1}^{s} P^{\rr}_{i, j}\]
	is a sum of $\{0,1\}$-valued random variables that are negatively correlated and whose opposite random variables are negatively correlated as well. So we may apply a Chernoff bound (Theorem~\ref{thm:chernoff1}) to obtain
	\begin{align*}
	\Pr\left( S(M_{[s] \times j}) \le (1- \delta) \frac{s}{k}n \right) &\le \exp \left( -\frac{\delta^2}{2} \frac{s}{k}n \right); \\
	\Pr\left(  S(M_{[s] \times j}) \ge (1+\delta) \frac{s}{k}n \right) &\le \exp \left( -\frac{\delta^2}{3} \frac{s}{k}n \right) .
	\end{align*}
	By a union bound, \[\Pr(\overline T) \le \sum_{j=1}^t \Pr(\overline T_j) \le 2t \cdot \exp \left( -\frac{\delta^2}{3} \frac{s}{k}n \right).\]

	Also consider the event $Q_j$ that all elements in $M_{[s] \times j}$ are nonzero and let $Q_{1..j}=\{Q_1, Q_2\ldots, Q_j\}$ be the event that all of them hold.
	We now prove two claims on some conditional probabilities.

	\begin{claim}\label{clm:probQiconditionally}
		For every $1 \le j \le t$,

		\[ \Pr(Q_j\mid T \setminus T_j, Q_{1..j-1}) \le \exp\left(- s \exp\left( -\left(1+\frac{\eps}{3}\right)\left(1+\delta\right)\frac{n}{k} \right)\right). \]
	\end{claim}
	\begin{claimproof}
		Fix $j$ between $1$ and $t$. Given a $k \times k$ permutation matrix $P$, let $\gamma(P)$ be the submatrix of $P$ induced by the columns indexed by $[t] \setminus j$. Likewise, given the random $n$- tuple $\mathcal{P}$ of $k \times k$ permutation matrices $(P^1, \ldots, P^n)$, let $\gamma(\mathcal{P})$ be the pointwise restriction $(\gamma(P^1), \ldots, \gamma(P^n))$. Let $W$ be the set of all possible tuples $\gamma(\mathcal{P})$ that satisfy $T \setminus T_j, Q_{1..j-1}$.
		Take any such partial realisation $w\in W$. From now on we additionally condition on the event $\gamma(\mathcal{P})=w$. To shorten notation, we will write
		$\Pr_w(\cdot) = \Pr( \cdot \mid \gamma(\mathcal{P})=w)$. Our goal is to upper bound  $\Pr_w(Q_j)$. (At the end of the proof we explain why this is sufficient.)  

		For each $i \in [s]$, let $a_i$ be equal to the sum
		$S(M_{i\times([t]\setminus j)})$ of the entries in $M_{i\times([t]\setminus j)}.$
		Let $\overline a$ be the mean of $a_i$ over all $i \in [s]$, and note that it is at least $(1-\delta)\frac{(t-1)n}{k}$ because $T \setminus T_j$ is satisfied.
		For every $i \in [s],$ let $Q_{i,j}$ be the event that $M_{i, j}\not=0.$

		We first show that the events $Q_{1,j},\ldots, Q_{s,j}$ are negatively correlated.
		For this, we have to prove that for any $I \subseteq [s],$ it is true that
		\begin{equation}\label{ineq:negativecorrelation}
			\Pr_w(\forall i \in I \colon Q_{i,j}) \le \prod_{i\in I} \Pr_w( Q_{i,j}).
		\end{equation}
		We prove~\eqref{ineq:negativecorrelation} by induction on $\lvert I \rvert.$ When $\lvert I \rvert\le 1$ the statement is trivially true.

		Let $I \subseteq [s]$ be a subset for which the statement is true and let $i' \in [s] \setminus I.$ We now prove the statement for $I'=I \cup \{i'\}.$
		We have

		\begin{equation}\label{ineq:negativecorrelation_inductionstep}	 \Pr_w(\forall i \in I \colon Q_{i,j}) \leq \Pr_w(\forall i \in I \colon Q_{i,j} \mid  \overline{Q_{i',j}} ), \end{equation}

\noindent
		because (informally) for every permutation $P^{\rr}$ and every $i\in I$, the probability of $P^\rr_{i,j}=1$ does not decrease if we add the condition that $P_{i',j}^{\rr}=0$.
		
		One can prove this formally as well, as we will now do.
	 Consider the values of $\rr$ for which $P^\rr_{i',j}=1$ is possible for an extension of $w$, i.e.~all $\rr$ for which row $i'$ in $\gamma(P^\rr)$ is a row with only zeros.
	 Let $R \subseteq [n]$ be the subset containing all such values $\rr.$
	 
	 For any subset $R_1 \subseteq R,$ let $\PP^{R_1}$ be the set of all possible tuples $\P'=\left(P^\rr_{[k] \times [t]}\right)_{1 \le \rr \le n}$ that extend $w$ (i.e.~with some abuse of notation, this is precisely when $\gamma(\P')=w$) and satisfy $P^\rr_{i',j}=1$ if and only if $\rr \in R_1$.
    The set $\PP^{R_1}$ helps us to specify the event that $M_{i',j}=|R_1|$.
	 
	 For a fixed nonempty $R_1 \subseteq R,$ consider the map $\phi \colon \PP^\emptyset \to \PP^{R_1}$ where $\phi(\P')$ is obtained from $\P'$ by considering, for each $r\in R_1$, the matrix $P^\rr_{[k] \times [t]}$ and replacing its $j^{th}$ column $P^\rr_{[k] ,j}$ with the unit vector $e_{i'}$, i.e.~setting $P^\rr_{i ,j} =\indicator{i=i'}$ for every $i \in [k].$
	 Note that $\phi$ is surjective and maps $(k-t)^{|R_1|}$ tuples to $1,$ i.e.~for every $\P'' \in \PP^{R_1}$, $\lvert \phi^{-1}(\P'')\rvert =(k-t)^{|R_1|}$.
	 Moreover, for every $\P'' \in \PP^{R_1}$, if $\P''$ satisfies $ Q_{i,j}$ for all $i \in I$, then every preimage $\P' \in \phi^{-1}(\P'')$ does.
	This implies that the probability that a random partial realisation in $\PP^{\emptyset}$ satisfies $Q_{i,j}$ for all $i \in I $ is at least as large as the probability that a random partial realisation in $\PP^{R_1}$ does.
To relate this to the probabilities that we are actually interested in, observe that for each $R_1 \subseteq R$, the probability that a random partial realisation in $\PP^{R_1}$ satisfies  $Q_{i,j}$ for all $i\in I$ is exactly equal to $\Pr^{R_1}:= \Pr_w\left( \forall i \in I \colon Q_{i,j} \mid P^{\rr}_{i',j}=1 \Leftrightarrow \rr\in R_1 \right)$. 
It follows that $\Pr^{\emptyset} \geq \Pr^{R_1}$ for all $R_1 \subseteq R$.
We conclude that
\[\Pr_w\left( \forall i \in I \colon Q_{i,j} \mid \overline {Q_{i',j}} \right) = \Pr^{\emptyset} \geq \sum_{R_1 \subseteq R} \Pr^{R_1} \cdot \Pr_w(P^{\rr}_{i',j}=1 \Leftrightarrow r \in R_1) =\Pr_w(\forall i \in I: Q_{i,j}).\]
	This proves~\eqref{ineq:negativecorrelation_inductionstep}.

		Now note that
		\begin{align*}
			\Pr_w(\forall i \in I \colon Q_{i,j})       & \le \Pr_w(\forall i \in I \colon Q_{i,j} \mid  \overline{ Q_{i',j} }) \\
			\iff \Pr_w(\forall i \in I \colon Q_{i,j} ) & \ge \Pr_w(\forall i \in I \colon Q_{i,j} \mid Q_{i',j} )              \\
			\iff\Pr_w(\forall i \in I' \colon Q_{i,j})  & \le 	\Pr_w(\forall i \in I \colon Q_{i,j}) \cdot \Pr_w(Q_{i',j}).
		\end{align*}
		This last expression is at most $ \prod_{i\in I'} \Pr_w(Q_{i,j})$ by the induction hypothesis, as desired. This concludes the proof of~\eqref{ineq:negativecorrelation}.

		So given $w$ with $S(M_{i\times([t]\setminus j)})=a_i$ for all $i \in [s]$, the probability $\Pr_w(Q_j )$ that none of the $s$ elements in $M_{[s]\times j}$ equals zero is at most $\prod_{i=1}^{s} \Pr_w(Q_{i,j}).$
		We can compute the latter as follows:

		\begin{align*}\Pr_w(Q_j) & \le \prod_{i=1}^{s} \Pr_w(Q_{i,j} )                                                                                                                   \\
                         & =\prod_{i=1}^{s} \left( 1- \Pr_w\left( P_{i,j}^{r}=0 \text{ for all } r \in [n] \right)\right)                                                        \\
                         & =\prod_{i=1}^{s} \left( 1- \prod_{r \in [n]}\Pr_w\left( P_{i,j}^{r}=0  \right)\right)                                                                 \\
                         & =\prod_{i=1}^{s} \left( 1- \prod_{r \in [n]  \text{ s.t. } S\left(P^r_{i \times ([t]\setminus j)}\right) = 0}\Pr_w\left( P_{i,j}^{r}=0 \right)\right) \\
                         & =\prod_{i=1}^{s} \left( 1- \left( 1-\frac{1}{k-(t-1)}\right)^{n-a_i}\right)                                                                           \\ &\le
              \exp\left(-\sum_{i=1}^{s} \left( 1-\frac{1}{k-(t-1)}\right)^{n-a_i}\right)                                                                                         \\
                         & \le \exp\left(- s \left( 1-\frac{1}{k-(t-1)}\right)^{n-\overline a }\right)                                                                           \\
                         & \le \exp\left(- s \left( 1-\frac{1}{k-(t-1)}\right)^{\left(1+\delta\right)\frac{k-(t-1)}{k}n }\right)                                                 \\
                         & \le \exp\left(- s \exp\left( -\left(1+\frac{\eps}{3}\right)\left(1+\delta\right)\frac{n}{k} \right)\right).
		\end{align*}

		Here we made use of $1-x \le \exp(-x)$ in the sixth line.
		We applied Jensen's inequality to the convex function $f(x)=\left( 1-\frac{1}{k-(t-1)}\right)^x$ in the seventh line.
		In the eighth line we used that $\overline a \ge (1-\delta)\frac{(t-1)}{k}n$ and $t-1 \le \frac{k}{2}$ to derive $n-\overline a \le n-(1-\delta)\frac{(t-1)}{k}n \le (1+\delta) \frac{k-(t-1)}{k}n$ and in the last line we used that $\left(1-\frac 1x\right)^x \ge \exp\left( -\left(1+\frac{\eps}{3}\right)\right)$ for $x=k-(t-1)\ge \frac k2$, which is sufficiently large in terms of $\eps.$

		Since this upper bound is true for any partial realisation $w\in W$, we conclude that
		\begin{align*}
			\Pr(Q_j\mid T \setminus T_j, Q_{1..j-1}) & =
			\sum_{w \in W}  \Pr( \gamma(\mathcal{P})=w \mid T \setminus T_j, Q_{1..j-1})
			\cdot \Pr_w(Q_j  )                                                                                                                                     \\
			                                         & \le \exp\left(- s \exp\left( -\left(1+\frac{\eps}{3}\right)\left(1+\delta\right)\frac{n}{k} \right)\right).\qedhere
		\end{align*}
	\end{claimproof}

	\begin{claim}\label{clm:probTiconditionally}
		For every $1 \le j \le t$, 	\[\Pr(\overline T_j\mid T \setminus T_j, Q_{1..j-1}) \le 2\exp \left( -\frac{\delta^2}{12}\left(1-\frac\delta3\right) \frac{s n}{k} \right) .\]
	\end{claim}
	\begin{claimproof}
		As in the proof of Claim~\ref{clm:probQiconditionally}, we fix $j$ and we condition on the event that the random tuple $\gamma(\mathcal{P})$ of permutation matrices restricted to the columns $[t]\setminus j$ is equal to a certain tuple $w\in W$, i.e.~such that $T \setminus T_j, Q_{1..j-1}$ are satisfied. We again fix such a realisation $w$ and succinctly write $\Pr_w(\cdot):= \Pr(\cdot \mid \gamma(\mathcal{P}) = w)$ for the conditional probability distribution.

		We will first estimate $\Pr_w(\overline T_j)$.
		Let $f_{\rr}=S(P^\rr_{[s] \times ([t]\setminus j)})$ for every $\rr \in [n]$, and note that we know this value deterministically because $P^\rr_{[s] \times ([t]\setminus j)}$ is a submatrix of $\gamma(P^\rr)$ and thus fully determined by $w$.
		Furthermore, for each $i\in [s]$, we have that $\Pr_w(P^{\rr}_{i,j}=1)$ equals $\frac{1}{k-(t-1)}$ if $S(P^{\rr}_{i,([t]\setminus j)})=0$, and equals $0$ otherwise.
		It follows that $\Exp_w(S(P^\rr_{[s] \times j}))= \frac{s-f_{\rr}}{k-(t-1)}$ and hence

		\begin{equation}\label{ineq:expect}
			\Exp_w(S(M_{[s] \times j}))=\sum_{\rr=1}^{n} \Exp_w(S(P^\rr_{[s] \times j}))
			=\frac{sn-\sum_{\rr=1}^{n} f_{\rr}}{k-(t-1)}.
		\end{equation}
		Since $w$ satisfies $T \setminus T_j$, we know $\sum_{\rr=1}^{n} f_{\rr} = S(M_{[s] \times ([t]\setminus j)})$ is between $(t-1) (1- \delta) \frac{sn}{k}$ and $(t-1) (1+ \delta) \frac{sn}{k}$.
		Evaluating these bounds in~\eqref{ineq:expect} and using the assumption $t \le \frac k4$, it follows that \[ (1- \delta/3) \frac{sn}{k}  \le \Exp_w \left[S(M_{[s] \times j}) \right] \le (1+ \delta/3) \frac{sn}{k}.\]

		Viewing $\left(P^\rr_{[k] \times [t]}\right)_{1 \le \rr \le n}$ as a random extension of $w$, we have that $S(P^\rr_{[s]\times j})$ and $S(P^{\rr'}_{[s]\times j})$ are independent\footnote{This independence is one of the reasons why we condition on $\gamma(\mathcal{P})$ being equal to a fixed realisation $w$, rather than conditioning on the full event $T \setminus T_j, Q_{1..j-1}$.} for $\rr \not= \rr'$, with respect to $\Pr_w(\cdot)$.
		Thus we can apply Theorem~\ref{thm:chernoff1} to the $\{0,1\}$-valued random variables $S(P^{1}_{[s] \times j}), \ldots, S(P^{n}_{[s]  \times j})$ to obtain
		\begin{align*}
		\Pr_w\left( S(M_{[s] \times j}) \le (1- \delta/2) (1- \delta/3) \frac{sn}{k}  \right) &\le \exp \left( -\frac{\delta^2}{8} (1- \delta/3) \frac{sn}{k} \right); \\
		\Pr_w\left(  S(M_{[s] \times j}) \ge (1+\delta/2)(1+\delta/3) \frac{sn}{k}  \right) &\le \exp \left( -\frac{\delta^2}{12} (1-\delta/3) \frac{sn}{k} \right) .
		\end{align*}
		Combining these two inequalities, and noting that $1-\delta \leq (1- \delta/2) (1- \delta/3) \leq (1+\delta/2)(1+\delta/3)  \leq 1+\delta$ due to the assumption $\delta \in[0,1]$, we obtain the desired bound:
		\[\Pr_w(\overline T_j) \le 2\exp \left( -\frac{\delta^2}{12} (1-\delta/3) \frac{sn}{k} \right).\]

		Finally we can conclude by writing $\Pr(\overline T_j\mid T \setminus T_j, Q_{1..j-1})$ as the linear combination $\sum_{w \in W}  \Pr( \gamma(\mathcal{P})=w \mid T \setminus T_j, Q_{1..j-1}) \cdot
			\Pr_w(\overline T_j )$, which is at most $\max_{w \in W} \Pr_w(\overline{T_j}).$
	\end{claimproof}

	We also prove the following claim.
	\begin{claim}\label{clm:relationProbabilities}
		Let $Q,A,T$ be any three events.
		Then \[\Pr(Q \mid A,T) \le \Pr(Q \mid A) + 2 \Pr( \overline T \mid A).\]
	\end{claim}
	\begin{claimproof}
		If $\Pr( \overline T \mid A) \ge 1/2$, the statement is trivially true.
		Note that \[\Pr(Q \mid A)=\Pr(Q \mid A,T) \Pr(T\mid A) + \Pr(Q\mid \overline T, A)\Pr(\overline T \mid A) \ge \Pr(Q \mid A,T) \Pr(T\mid A).\]
		Since $\frac1{1-x}\le 1+2x$ for any $x \le \frac 12$, $\Pr( \overline T \mid A) < 0.5$ and $
			\Pr(Q \mid A,T) \le \frac{ \Pr(Q \mid A)}{1-\Pr(\overline T \mid A) }$ imply that
		$		\Pr(Q \mid A,T)  \le \Pr(Q \mid A) + 2 \Pr(Q \mid A)\Pr(\overline T \mid A) \le \Pr(Q \mid A) + 2 \Pr( \overline T \mid A).$
	\end{claimproof}

	Applying Claim~\ref{clm:relationProbabilities} with the formulas in Claims~\ref{clm:probQiconditionally} and~\ref{clm:probTiconditionally} where $(Q,A,T)$ in Claim~\ref{clm:relationProbabilities} is chosen to be $(Q_j, (T\setminus T_j)\cap Q_{1..j-1}, T_j),$
	we deduce that for every $1\le j \le t$ \[\Pr(Q_j\mid T\cap Q_{1..j-1})\le 4 \exp \left( -\frac{\delta^2}{12}\left(1-\frac\delta3\right) n \frac{s}{k} \right) + \exp\left(- s \exp\left( -\left(1+\frac{\eps}3\right)\left(1+\delta\right)\frac{n}{k} \right)\right).\]

	Finally note that
	$
		\Pr(Q_{1..t}\mid T)= \Pr(Q_1 \mid T) \Pr(Q_2 \mid Q_1,T) \cdots \Pr(Q_t \mid Q_{1..t-1},T) $
	and so the conclusion follows from $\Pr(Q_{1..t}) \le \Pr(Q_{1..t}\mid T) + \Pr(\overline T).$
\end{proof}



\paragraph{Remark.}
	The case $t=1$ of Lemma~\ref{lem:low_probability_nonzero_l*msubmatrix_sumpartialpermutations_l} is much simpler. In  this case the fixed $s \times t$ submatrix of $M$ has only one column, and within that column the events $Q_{1,1}, Q_{1,2}, \ldots, Q_{s,1}$ are negatively correlated, so $\Pr(Q)$ is at most $ \left( 1- \left( 1-\frac{1}{k} \right)^{n} \right)^{s}$. In contrast, for $t\geq 2$ a more involved argument is required because (as discussed before the proof) even $Q_{1,1},Q_{1,2},Q_{2,1}$ and $Q_{2,2}$ are not necessarily negatively correlated.

\begin{proof}[Proof of Lemma~\ref{cor:low_probability_zerotransversal}]
	By the Frobenius--K\"onig theorem (Lemma~\ref{le:distinctrepr}), no such transversal in $M$ exists if and only there exist $s, t$ with $s +t=k+1$ and a $s \times t$-submatrix all of whose entries are nonzero.
	For every such choice of $s$ and $t$ there are $\binom{k}{s}\binom{k}{t}$ possible $s \times t$-submatrices. The case where $s \le t$ is similar to the case where $s\ge t$, by switching the two. Note that $t \le s$ and $s +t=k+1$ imply $s \ge \frac k2.$
	When $\frac k4 <t$, the probability that a $s \times t$-submatrix has only nonzero entries is obviously at most the probability that a $s \times \frac k 4$-submatrix has only nonzero entries.

	We may now use the probability computed in Lemma~\ref{lem:low_probability_nonzero_l*msubmatrix_sumpartialpermutations_l} with the choice $\delta= 1/\sqrt[4]{\log n}$.
	Noting that $\sum_{s,t} \binom{k}{t}\binom{k}{s}=4^k$, the sum of the corresponding first terms is bounded by
	$4^k k \exp\left( -\frac{n}{6\sqrt{\log n}} \right).$
	For the second terms, we first compute that for $s\ge \frac k2$,
	\begin{align*}
		 4& \exp \left( -\frac{\delta^2}{12}\left(1-\frac\delta3\right) n \frac{s}{k} \right) + \exp\left(- s \exp\left( -\left(1+\frac{\eps}3\right)\left(1+\delta\right)\frac{n}{k} \right)\right) \\
		 & \le
		\exp\left(- \frac k2 \exp\left( -\frac{\log n}{1+\eps/2} \right)\right)+4\exp\left( -\frac{n}{25 \sqrt{\log n}}\right)                                                                        \\
		 & \le \exp \left( -\frac{n^{\frac{\eps /2}{1+\eps /2}}}{2\log n} \right) + 4\exp\left( -\frac{n}{25 \sqrt{\log n}}\right)                                                                    \\
		 & \le \exp \left( -\frac{n^{\frac{\eps /2}{1+\eps /2}}}{2\log n} \right) \left( 1+ 4 \exp\left( -\frac{  n }{ 26\sqrt{\log n}}\right)\right)                                                 \\
		 & \le \exp \left( -n^{\eps /3} \right).
	\end{align*}

	Noting that $\binom{k}{t}\binom{k}{k+1-t}\le k^{2t-1}$, we have the following upper bound for the sum of the corresponding second terms from Lemma~\ref{lem:low_probability_nonzero_l*msubmatrix_sumpartialpermutations_l}:
	\[2\sum_{t=1}^{k/4} \left(k^2 \exp \left( -n^{\eps /3}  \right) \right)^t +
		4^k  \exp \left( -n^{\eps /3}  \frac{k}{4} \right)<2.5k^2\exp \left( -n^{\eps /3}  \right).\]
	We conclude by noting that
	\begin{align*}
	2.5k^2\exp \left( -n^{\eps /3}  \right)+4^k k \exp\left( -\frac{n}{6\sqrt{\log n}} \right) &< 3k^2\exp \left( -n^{\eps /3}  \right).\qedhere
	\end{align*}
\end{proof}

\section{Concluding remarks}\label{sec:conclusion}

We have set the stage for this natural fusion between two classic notions in (extremal) graph theory: packing and colouring.
We were generous throughout with problems for further investigation, most importantly Conjecture~\ref{conj:main}, which we might audaciously refer to as the \emph{List Packing Conjecture}. Regardless of its truth in general, Conjecture~\ref{conj:main} may be specialised to various graph classes for many research possibilities. We highlight three fundamental classes where it is, in our view, most tempting to push for further progress.
\begin{itemize}
\item
\emph{Planar graphs}. We do not yet have any constructions to rule out the possibility that $\chi^\star_\ell(G) \le 5$ for all planar $G$. What is the optimal value?
\item
\emph{Line graphs}. Based on the List Colouring Conjecture, we surmise for every $\eps>0$ that $\chi^\star_\ell(G) \le (1+\eps)\omega$ for every line graph $G$ with clique number $\omega\ge \omega_0$.
Due to its connection to Latin squares, here even the case where $G$ is the line graph of the complete bipartite graph $K_{\omega,\omega}$ is enticing to narrow in on.
\item
\emph{Random graphs}. Does it hold that $\chi^\star_\ell(G_{n,1/2}) \le (1+o(1))n/(2\log_2 n)$ a.a.s.?
Related to this question, one might wonder if the $\log n$ factor in Theorem~\ref{thm:fractional} could be improved to one of order $\log (n/\chi_f(G))$, for this would immediately imply an upper bound for $\chi^\star_\ell(G_{n,1/2})$ of order $\frac{n\log\log n}{\log n}$ a.a.s.
\end{itemize}

\noindent
We also consider it natural to pursue results along the continuum between list packing and strong colouring as discussed in Subsection~\ref{sub:continued}. Here it seems to us that Conjecture~\ref{conj:CaKa} is within reach with the methods of Section~\ref{sec:bip}, but we leave this to future investigation.

\section*{Acknowledgements}
We thank the anonymous referees for their comments.

 {\small
\paragraph{Open access.} For the purpose of open access, a CC BY public copyright licence is applied to any Author Accepted Manuscript (AAM) arising from this submission.}

\bibliographystyle{abbrv}
\bibliography{listpack}

\end{document}